\documentclass[a4paper,11pt]{amsart}

\usepackage{amssymb}
\usepackage{amscd}
\usepackage{amsthm}
\usepackage{amsmath}
\usepackage[all]{xy}
\usepackage{extarrows}
\usepackage{rotating}
\usepackage{tikz-cd}
\usepackage{leftidx}
\usepackage{bm}
\usepackage{dynkin-diagrams}
\usepackage{mathrsfs}
\usepackage{rotating}
\usepackage{upgreek}

\usepackage[pagebackref=true]{hyperref}

\renewcommand*{\backref}[1]{}
\renewcommand*{\backrefalt}[4]{[{\tiny%
		\ifcase #1 Not cited.%
		\or Cited on page~#2.%
		\else Cited on pages #2.%
		\fi%
	}]}

\DeclareUnicodeCharacter{2212}{-}
\DeclareUnicodeCharacter{0008}{~}
\DeclareUnicodeCharacter{2011}{~}

\def\mathbi#1{\textbf{\em #1}}

\renewenvironment{proof}{{ \textbf{Proof}.}}{\qed}
\newtheorem{Thm}{Theorem}[section]
\newtheorem{Lem}[Thm]{Lemma}
\newtheorem{Def}[Thm]{Definition}
\newtheorem{Cor}[Thm]{Corollary}
\newtheorem{Prop}[Thm]{Proposition}
\newtheorem{Ex1}[Thm]{Example}
\newtheorem{Rem1}[Thm]{Remark}

\newcommand{\cone}{\mathrm{Cone}}
\newcommand{\Hom}{\mathrm{Hom}}

\newcommand{\add}{\mathrm{add}}
\newcommand{\per}{\mathrm{per}}
\newcommand{\Ext}{\mathrm{Ext}}
\newcommand{\End}{\mathrm{End}}

\newcommand{\bijar}[1][]{%
	\ar[#1]
	\ar@<0.7ex>@{}[#1]|-*[@]{\sim}}

\newcommand{\pvd}{\mathrm{pvd}}

\newcommand{\cten}{\widehat{\otimes}}

\newcommand{\Bd}{\mathrm{Bd}}
\newcommand{\PcAlgc}{\mathrm{PcAlgc}(R)}
\newcommand{\PcAlg}{\mathrm{PcAlg}(R)}
\newcommand{\bL}{\mathbf{L}}

\newenvironment{Rem}{\begin{Rem1}\rm}{\end{Rem1}}
\newenvironment{Ex}{\begin{Ex1}\rm}{\end{Ex1}}

 \normalbaselines
\newcommand{\Si}{\Sigma}

\newcommand{\ten}{\otimes}
\newcommand{\lten}{\overset{\boldmath{L}}{\ten}}
\newcommand{\lcten}{\widehat{\ten}^{\mathbf{L}}}

\newcommand{\cof}{\mathrm{cof}}
\newcommand{\RHom}{\mathrm{\mathbf{R}Hom}}
\newcommand{\ra}{\rightarrow}

\newcommand{\iso}{\xrightarrow{_\sim}}

\newcommand{\id}{\mathbf{1}}

\newcommand{\ca}{{\mathcal A}}
\newcommand{\cb}{{\mathcal B}}
\newcommand{\cc}{{\mathcal C}}
\newcommand{\cd}{{\mathcal D}}

\newcommand{\cf}{{\mathcal F}}
\newcommand{\cg}{{\mathcal G}}
\newcommand{\ch}{{\mathcal H}}

\newcommand{\cm}{{\mathcal M}}

\newcommand{\ct}{{\mathcal T}}

\begin{document}
	
	\title{Categorification of ice quiver mutation}
	
	

	\author{Yilin WU}
	\address{
		School of Mathematical Sciences\\
		University of Science and Technology of China\\
		Hefei 230026, Anhui\\
		P. R. China
		and
		Universit\'e Paris Cit\'e\\
		UFR de Math\'ematiques\\
		Institut de Math\'ematiques de Jussieu--Paris Rive Gauche, IMJ-PRG \\
		B\^{a}timent Sophie Germain\\
		75205 Paris Cedex 13\\
		France
	}
	\email{wuyilinecnuwudi@gmail.com\\
		yilin.wu@imj-prg.fr}
	
	
	\dedicatory{}
	
	
	\begin{abstract}
		 In 2009, Keller and Yang categorified quiver mutation by interpreting it in terms of equivalences between derived categories. Their approach was based on Ginzburg’s Calabi--Yau algebras and on Derksen--Weyman--Zelevinsky’s mutation of quivers with potential. Recently, Pressland has generalized mutation of quivers with potential to that of ice quivers with potential. In this paper, we show that his rule yields derived equivalences between the associated relative Ginzburg algebras, which are special cases of Yeung’s deformed relative Calabi--Yau completions arising in the theory of relative Calabi--Yau structures due to Toën and Brav--Dyckerhoff. We illustrate our results on examples arising in the work of Baur--King--Marsh on dimer models and cluster categories of Grassmannians. We also give a categorification of mutation at frozen vertices as it appears in recent work of Fraser--Sherman-Bennett on positroid cluster structures.
	\end{abstract}
	
	\maketitle
	\tableofcontents
	
	\section{Introduction}
	Almost 20 years ago, Fomin and Zelevinsky~\cite{fominClusterAlgebrasFoundations2002} invented cluster algebra in order to create a combinatorial framework for the study of canonical bases in quantum groups and the study of total positivity in algebraic groups. A cluster algebra is a commutative algebra with a distinguished family of generators, called cluster variables, displaying special combinatorial properties. To construct it, we start with a seed $$ (X=(x_{1},\ldots,x_{n},x_{n+1},\ldots,x_{m}),B) ,$$ consisting, by definition, of a set $ X $ which freely generates an ambient field $ \cf=\mathbb{Q}(x_{1},\ldots,x_{m}) $, and an integer $ m\times n $ matrix $ B=(b_{ij}) $ such that the principal part $ B'=(b_{ij})_{1\leqslant i\leqslant n,1\leqslant j\leqslant n} $ is skew symmetric. Or, we can  instead of the matrix $ B $ use a finite quiver $ Q $ with vertices $ 1,2,\ldots,m $, and without oriented cycles of length 1 or 2. For each $ i=1,\ldots,n $, the mutation $ \mu_{i}(X,Q)=(X',Q') $ is defined by first replacing $ x_{i} $ with another element $ x^{*}_{i} $ in $ \cf $ according to a specific rule which depends upon both $ (x_{1},\ldots,x_{m}) $ and $ Q $. Then we get a new free generating set $ X'=(x_{1},\ldots,x_{i-1},x^{*}_{i},x_{i+1},\ldots,x_{n},x_{n+1},\ldots,x_{m}) $. The mutated quiver $ \mu_{i}(Q)=Q' $ is obtained from $ Q $ by applying a certain combinatorial rule depending on $ i $ to the arrows of $ Q $. This yields the new seed $ (X',Q') $. We continue applying $ \mu_{1},\ldots,\mu_{n} $ to the new seed to get further seeds. The $ m $-element sets $ X'' $ occurring in seeds $ (X'',Q'') $ are called \emph{clusters}, and the elements in the
	clusters are called \emph{cluster variables}. The cluster variables $ x_{n+1},\ldots,x_{m} $ cannot be mutated, these are called \emph{frozen variables}. The associated cluster algebra is the subalgebra of the function field $ \cf $ generated by all cluster variables. The basic combinatorial construction in the definition of a cluster algebra is quiver mutation. It is an elementary operation on quivers. In mathematics, it first appeared in Fomin-Zelevinsky's definition of cluster algebras. In physics, it appeared a few years before in Seiberg duality~\cite{seibergElectricmagneticDualitySupersymmetric1995}. Beyond the first simple examples, iterated quiver mutation becomes extremely complex, which has been one of the main motivations for interpreting it at the categorical level, where more conceptual tools become available.
	
	Let $ Q $ be a finite quiver and $ i $ a source of $ Q $, i.e.$ \, $a vertex without incoming arrows. Let $ Q' $ be the 
	mutation of $ Q $ at $ i $, i.e.$ \, $the quiver obtained from $ Q $ by reversing all the arrows going out from $ i $. Let $ k $ be a field, $ kQ $ the path algebra of $ Q $ and $ \cd(kQ) $ the derived category of the category of all right $ kQ $-modules. For a vertex $ j $ of $ Q' $ respectively $ Q $, let $ P'_{j} $ respectively $ P_{j}$ be the projective indecomposable associated with the vertex $ j $.
	Then Bernstein–Gelfand–Ponomarev's~\cite{bernsteinCoxeterFunctorsGabriel1973} main result reformulated in terms of derived categories following Happel~\cite{happelDerivedCategoryFinitedimensional1987} says that there is a canonical triangle equivalence 
	$$ F:\cd(kQ')\ra\cd(kQ) $$ which takes $ P'_{j} $ to $ P_{j} $ for $ j\neq i $ and $ P'_{i} $ to the cone over the morphism $$ P_{i}\ra\bigoplus P_{j} $$ whose components are the left multiplications by all arrows going out from $ i $. This gives a categorical interpretation of the mutation at a \emph{source} $ i $.
	
	Keller and Yang~\cite{kellerDerivedEquivalencesMutations2011} obtained an analogous result for the mutation of a quiver with potential $ (Q,W) $ at an \emph{arbitrary} vertex $ i $, where the role of the quiver with reversed arrows is played by the quiver with potential $ (Q', W') $ obtained from $ (Q, W) $ by mutation at $ i $ in the sense of Derksen--Weyman--Zelevinsky~\cite{derksenQuiversPotentialsTheir2008}. The role of the derived category $ \cd(kQ) $ is now played by the derived category $ \cd(\bm\Gamma) $ of the complete differential graded algebra $ \bm\Gamma=\bm\Gamma(Q,W) $ associated with $ (Q, W) $ by Ginzburg~\cite{ginzburgCalabiyauAlgebras2006a}. Their result is analogous to but not a generalization of Bernstein-Gelfand-Ponomarev’s since even if the potential $ W $ vanishes, the derived category $ \cd(\bm\Gamma) $ is not equivalent to $ \cd(kQ) $. 
	
	Recently, Pressland~\cite{presslandMutationFrozenJacobian2020} has generalized mutation of quivers with potential to that of \emph{ice} quivers with potential. In this article, our aim is to ‘categorify’ his ice quiver mutation by interpreting it in terms of equivalences between derived categories. We show that his rule yields derived equivalences between the associated complete \emph{relative} Ginzburg algebras. The main new ingredient in our approach is the complete relative Ginzburg algebra $ \bm\Gamma_{rel}(Q,F,W) $ associated to an ice quiver with potential $ (Q,F,W) $, where $ Q $ is a finite quiver, $ F $ is a subquiver of $ Q $ and $ W $ is a potential on $ Q $. This differential graded (=dg) algebra is connective, i.e.$ \, $concentrated in non positive degrees. It can be viewed as a special case of Yeung's deformed relative Calabi--Yau completion~\cite{yeungRelativeCalabiYauCompletions2016} arising in the theory of relative Calabi--Yau structures due to To\"{e}n~\cite{toenDerivedAlgebraicGeometry2014a} and Brav--Dyckerhoff~\cite{bravRelativeCalabiYau2019}. In particular, when the frozen subquiver $ F $ is empty, we recover the main result of Keller and Yang.

	Unexpectedly, Fraser and Sherman-Bennett have very recently discovered a construction of mutation at \emph{frozen} vertices in their study~\cite{fraserPositroidClusterStructures2020} of cluster structures on positroid varieties. Let $ v $ be a frozen vertex. Suppose that $ v $ is a source vertex in $ F $ such that there are no unfrozen arrows with source $ v $, or $ v $ is a sink vertex in $ F $ such that there are no unfrozen arrows with target $ v $. Our second main result says that the mutation at the frozen vertex $ v $ is ‘categorified' by the twist (respectively inverse twist) functor $ t_{S_{v}} $ (respectively $ t^{-1}_{S_{v}} $ ) with respect to the $ 2 $-spherical object $ S_{v} $ (the simple module at vertex $ v $) in the derived category of the complete derived preprojective algebra $ \bm\Pi_{2}(F) $. In \cite{wuyilinMutationsFrozenVertices}, we will show how suitable compositions of mutations at frozen vertices can be decategorified into quasi cluster isomorphisms (see~\cite{fraserQuasihomomorphismsClusterAlgebras2016}).
	
	\subsection*{Plan of the paper and main results}
	
	In section~\ref{Section2}, we recall some basic definitions and properties about (pseudocompact) dg algebras and their (pseudocompact) derived categories. We also discuss the homological invariants in the pseudocompact setting. In Section~\ref{Section3}, we give the Definition~\ref{Def:deformed CY completion} of the deformed relative Calabi--Yau completion for a pseudocompact dg algebra. Then we describe how to extend some of the results of \cite{yeungRelativeCalabiYauCompletions2016} and \cite{bozecRelativeCriticalLoci2020} to the pseudocompact setting (see Theorem~\ref{Thm:Relative defomed CY completion has a canonical left CY}). Let $ (Q,F) $ be an ice quiver. Let $ v $ be an unfrozen vertex such that no loops or 2-cycles are incident with $ v $. In section~\ref{Section4}, we recall the definition of Pressland's combinatorial mutation~\cite{presslandMutationFrozenJacobian2020} $ \mu_{v}^{P}(Q,W) $ of $ (Q,F) $ at vertex $ v $. Let $ W $ be a potential on $ Q $. Then we recall Pressland's mutation $ \mu_{v}(Q,F,W) $ of $ (Q,F,W) $ at $ v $. By a result of Pressland (see~\cite[Proposition 4.6]{presslandMutationFrozenJacobian2020}), if the underlying ice quiver of $ \mu_{v}(Q,F,W) $ has no $ 2 $-cycles containing unfrozen arrows, then it agrees with $ \mu^{P}_{v}(Q,F) $.

	For any ice quiver with potential $ (Q,F,W) $, we define the associated \emph{complete relative Ginzburg algebra} $ \bm{\Gamma}_{rel}(Q,F,W) $ (see Definition~\ref{Def:Relative Ginzburg algebra}). Via the relative deformed 3-Calabi--Yau completion of $ G:\widehat{kF}\hookrightarrow \widehat{kQ} $ with respect the potential $ W $, we get the \emph{Ginzburg functor} 
	$$ \bm{G}_{rel}:\bm{\Pi}_{2}(F)\ra\bm{\Gamma}_{rel}(Q,F,W) .$$
	By Theorem~\ref{Thm:Relative defomed CY completion has a canonical left CY}, the Ginzburg functor $ \bm{G}_{rel} $ has a canonical left 3-Calabi--Yau structure in the sense of Brav-Dyckerhoff~\cite{bravRelativeCalabiYau2019}. In Section~\ref{Section5}, we show that relative Calabi--Yau completion is compatible with localization functors and Morita functors (see~Theorem~\ref{Thm:Derived equivalence}). 
	
	Let $ \tilde{\mu}_{v}(Q,F,W)=(Q',F,W') $ be the \emph{pre-mutation} (see Definition~\ref{Def:Algebraic mutation}) of $ (Q,F,W) $ at vertex $ v $. Applying Theorem~\ref{Thm:Derived equivalence} in a special case (see Subsection~\ref{Subsection:Derived equivalences}), we get our main results.
	
	\begin{Thm}(Theorem~\ref{Thm:equivalence non-frozen})
		Let $ \bm\Gamma_{rel}=\bm\Gamma_{rel}(Q,F,W) $ and $ \bm\Gamma'_{rel}=\bm\Gamma_{rel}(Q',F,W') $ be the complete Ginzburg dg algebras associated to $ (Q,F,W) $ and $ \tilde{\mu}_{v}(Q,F,W)=(Q',F,W') $ respectively. For a vertex $ i $, let $ \bm\Gamma_{i}=e_{i}\bm\Gamma_{rel} $ and $ \bm\Gamma'_{i}=e_{i}\bm\Gamma'_{rel} $.
		\begin{itemize}
			\item[a)] There is a triangle equivalence
			\begin{align*}
				\xymatrix{
					\Phi_{+}:\mathcal{D}(\bm\Gamma'_{rel})\ar[r]&\mathcal{D}(\bm\Gamma_{rel}),
				}
			\end{align*}
			which sends $ \bm\Gamma'_{j} $ to $ \bm\Gamma_{j} $ for $ j\neq v $ and to the cone over the morphism
			$$ \bm\Gamma_{v}\ra\bigoplus_{\alpha}\bm\Gamma_{t(\alpha)} $$ for $ j=v $, where we have a summand $ \bm\Gamma_{t(\alpha)} $ for each arrow $ \alpha $ of $ Q $ with source $ v $ and the corresponding component of the map is the left multiplication by $ \alpha $. The functor $	\Phi_{+} $ restricts to triangle equivalences from $ \per(\bm\Gamma'_{rel})$ to $ \per(\bm\Gamma_{rel}) $ and from $ \pvd(\bm\Gamma'_{rel}) $ to $ \pvd(\bm\Gamma_{rel}) $.
			\item[b)] Let $ \bm\Gamma^{red}_{rel} $ respectively $ \bm\Gamma'^{red}_{rel} $ be the complete Ginzburg dg algebra associated with the reduction of $ (Q,F,W) $ respectively the reduction $ \mu_{v}(Q,F,W)=(Q'',F'',W'') $ of $ \tilde{\mu}_{v}(Q,F,W) $. Then functor $ \Phi_{+} $ yields a triangle equivalence
			\begin{align*}
				\xymatrix{
					\Phi_{+}^{red}:\mathcal{D}(\bm\Gamma'^{red}_{rel})\ar[r]&\mathcal{D}(\bm\Gamma^{red}_{rel}),
				}
			\end{align*}
			which restricts to triangle equivalences from $ \per(\bm\Gamma'^{red}_{rel}) $ to $ \per(\bm\Gamma^{red}_{rel}) $ and from $ \pvd(\bm\Gamma'^{red}_{rel}) $ to $ \pvd(\bm\Gamma^{red}_{rel}) $.
			\item[c)] We have the following commutative diagram
			\[
			\begin{tikzcd}
				&&\cd(\bm\Gamma'_{rel})\arrow[dd,"\Phi_{+}"]\\
				\cd(\bm\Pi_{2}(F))\arrow[urr,"(\bm{G}'_{rel})^{*}"]\arrow[drr,swap,"(\bm{G}_{rel})^{*}"]&&\\
				&&\cd(\bm\Gamma_{rel}).
			\end{tikzcd}
			\]
			\item[d)] Since the frozen parts of $ \mu_{v}(Q,F,W)=(Q'',F'',W'') $ and of $ \tilde{\mu}_{v}(Q,F,W)=(Q',F,W') $ 
			only differ in the directions of the frozen arrows, we have a canonical isomorphism between $ \Pi_{2}(kF) $ and $ \Pi_{2}(kF'') $. It induces a canonical triangle equivalence
			$$ \mathrm{can}:\cd(\bm\Pi_{2}(kF))\ra\cd(\bm\Pi_{2}(kF'')) .$$ Then the following diagram commutes up to isomorphism
			\[
			\begin{tikzcd}
				\cd(\bm\Pi_{2}(F''))\arrow[rr,"(\bm{G}'_{rel})^{*}"]\arrow[dd,"\mathrm{can}^{-1}",swap]&&\cd(\bm\Gamma'^{red}_{rel})\arrow[dd,"\Phi_{+}^{red}"]\\
				\\
				\cd(\bm\Pi_{2}(F))\arrow[rr,swap,"(\bm{G}_{rel})^{*}"]&&\cd(\bm\Gamma^{red}_{rel}).
			\end{tikzcd}
			\]
		\end{itemize}	
	\end{Thm}
	
	\bigskip
	We then define the associated \emph{boundary dg algebra} $ \Bd(Q,F,W) $ (see Definition~\ref{Def:boundary dg algebra}) as 
	$$ \Bd(Q,F,W)=\mathrm{REnd}_{\bm\Gamma_{rel}(Q,F,W)}((\bm{G}_{rel})^{*}(\bm{\Pi}_{2}(F)))\simeq e_{F}\bm\Gamma_{rel}(Q,F,W)e_{F} ,$$ where $ e_{F}=\sum_{i\in F}e_{i} $ is the sum of idempotents corresponding to the frozen vertices. Corollary~\ref{Cor:The boundary dg algebra is invariant under mutations} shows that the boundary dg algebra is invariant under mutations at unfrozen vertices. By using this Corollary, we illustrate our results on examples arising in the work of Baur--King--Marsh on dimer models and cluster categories of Grassmannians (see Example~\ref{Ex:Postnikov diagram}). 
	
	In the last section, we study the categorification of mutation at a frozen vertex. Let $ (Q,F,W) $ be an ice quiver with potential. Let $ v $ be a frozen vertex. Suppose that $ v $ satisfies the following conditions
	\begin{itemize}
		\item[1)] $ v $ is a source vertex in $ F $ such that there are no unfrozen arrows with source $ v $, or
		\item[2)] $ v $ is a sink vertex in $ F $ such that there are no unfrozen arrows with target $ v $.
	\end{itemize}
	In this situation, we define (see Definition~\ref{Def:Combinatorial ice mutations}) the mutation of $ (Q,F) $ at the frozen vertex $ v $ by using the same mutation rule as that defined by Pressland for mutation at unfrozen vertices. Then we also give the definition of the mutation of an ice quiver with potential at the frozen vertex $ v $ (see Definition~\ref{Def:algebraic ice mutation}).
	
	If $ v $ is a source vertex in $ F $, our result shows that the mutation at $ v $ is ‘categorified' by the inverse twist functor  $ t^{-1}_{S_{v}} $ with respect to the $ 2 $-spherical object $ S_{v} $ (the simple module at vertex $ v $) in $ \cd(\bm\Pi_{2}(F)) $. Let us make this more precise. Write $ (Q',F',W')=\tilde{\mu}_{v}(Q,F,W) $. Let $ \bm\Gamma_{rel}=\bm\Gamma_{rel}(Q,F,W) $ and $ \bm\Gamma'_{rel}=\bm\Gamma_{rel}(Q',F',W') $ be the complete relative Ginzburg dg algebras associated to $ (Q,F,W) $ and $ (Q',F',W') $ respectively.
	\begin{Thm}(Theorem~\ref{Thm:categorification of ice mutaion1})
		We have a triangle equivalence
		$$ \Psi_{+}:\cd(\bm\Gamma'_{rel})\ra\cd(\bm\Gamma_{rel}) ,$$
		which sends the $ \bm\Gamma'_{i} $ to $ \bm\Gamma_{i} $ for $ i\neq v $ and $ \bm\Gamma'_{v} $ to the cone
		$$ \cone(\bm\Gamma_{v}\ra\bigoplus_{\alpha}\bm\Gamma_{t(\alpha)}) ,$$ where we have a summand $ \bm\Gamma_{t(\alpha)} $ for each arrow $ \alpha $ of $ F $ with source $ v $ and the corresponding component of the map is the left multiplication by $ \alpha $. The functor $ \Psi_{+} $ restricts to triangle equivalences from $ \per(\bm\Gamma'_{rel}) $ to $ \per(\bm\Gamma_{rel}) $ and from $ \pvd(\bm\Gamma'_{rel}) $ to $ \pvd(\bm\Gamma_{rel}) $. Moreover, the following square commutes up to isomorphism
		\begin{equation}\label{Diagram:left ice mutation}
			\begin{tikzcd}
				\cd(\bm\Pi_{2}(F'))\arrow[r]\arrow[d,swap,"\mathrm{can}"]&\cd(\bm\Gamma'_{rel})\arrow[d,"\Psi_{+}"]\\
				\cd(\bm\Pi_{2}(F))\arrow[d,swap,"t^{-1}_{S_{v}}"]&\cd(\bm\Gamma_{rel})\arrow[d,equal]\\
				\cd(\bm\Pi_{2}(F))\arrow[r,swap]&\cd(\bm\Gamma_{rel}),
			\end{tikzcd}
		\end{equation}
		where $\mathrm{can}$ is the canonical functor induced by an identification between $ \bm\Pi_{2}(F') $ and $ \bm\Pi_{2}(F) $ and $ t^{-1}_{S_{v}} $ is the twist inverse functor with respect to the $ 2 $-spherical object $ S_{v} $, which gives rise to a triangle
		$$ t^{-1}_{S_{v}}(X)\ra X\ra\Hom_{k}(\RHom_{\bm\Pi_{2}(F)}(X,S_{v}),S_{v})\ra\Si t^{-1}_{S_{v}}(X) $$ for each object $ X $ of $ \cd(\bm\Pi_{2}(F)) $.
	\end{Thm}
	
	Dually, if $ v $ is a sink vertex in $ F $, the mutation at $ v $ is ‘categorified' by the twist functor  $ t_{S_{v}} $ with respect to the $ 2 $-spherical object $ S_{v} $ in $ \cd(\bm\Pi_{2}(F)) $.
	

	\section*{Acknowledgments}
	
	The author is supported by the China Scholarship Council (CSC, grant number: 201906140160), Innovation Program for Quantum Science and Technology (2021ZD0302902), Natural Science Foundation
	of China (Grant Nos. 12071137) and STCSM (No. 18dz2271000). He would like to thank his PhD supervisor Bernhard Keller for his his guidance, patience and kindness. He would also like to thank his PhD co-supervisor Guodong Zhou for the constant support and encouragement during his career in mathematics. He is grateful to Xiaofa Chen for many interesting discussions and useful comments. He is very grateful to the anonymous referee for suggesting various improvements.

	\section{Preliminaries}\label{Section2}
	In this section, we recall some basic definitions and properties of (pseudocompact) dg algebras and their (pseudocompact) derived categories. Then we discuss homological invariants in the pseudocompact setting. Our main references are 
	\cite{kellerDerivingDGCategories1994a}, \cite{kellerDifferentialGradedCategories2006}, \cite[Appendix A]{kellerDerivedEquivalencesMutations2011} and \cite{vandenberghCalabiYauAlgebrasSuperpotentials2015}.
	\subsection{Derived categories of dg algebras}
	Let $ k $ be a commutative ring. 
	\begin{Def}\rm
		A \emph{differential graded $ k $-algebra} (or simply dg $ k $-algebra) is a graded $ k $-algebra $ A=\bigoplus_{n\in\mathbb{Z}}A^{n} $ equipped with a $ k $-linear homogeneous map $ d_{A}:A\ra A $ of degree 1 such that $ d_{A}^{2}=0 $ and the graded Leibniz rule $ d_{A}(ab)=(d_{A}a)b+(-1)^{n}ad_{A}(b) $ holds, where $ a\in A^{n} $ and $ b\in A $. The map $ d_{A} $ is called the \emph{differential} of $ A $.
	\end{Def}
	
	We can view an ordinary $ k $-algebra as a dg $ k $-algebra concentrated in degree $ 0 $ whose differential is trivial. A graded $ k $-algebra can be viewed as a dg $ k $-algebra with the zero differential.

	Let $ A $ be a dg $ k $-algebra with differential $ d_{A} $.
	\begin{Def}\rm
		A \emph{right dg module} over $ A $ is a graded right $ A $-module $ M=\bigoplus_{n\in\mathbb{Z}}M^{n} $ equipped with a $ k $-linear homogeneous map $ d_{M}:M\ra M $ of degree 1 such that $ d_{M}^{2}=0 $ and the graded Leibniz rule $$ d_{M}(ma)=d_{M}(m)a+(-1)^{n}md_{A}(a) $$ holds for all $ m\in M^{n} $ and $ a\in A $. The map $ d_{M} $ is called the \emph{differential} of $ M $.
	\end{Def}
	
	Given two dg $ A $-modules $ M $ and $ N $, we define the \emph{morphism complex} to be the graded $ k $-vector space $ \ch om_{A}(M,N) $ whose $ i $-th component $ \ch om^{i}_{A}(M,N) $ is the subspace of the product $ \prod_{j\in\mathbb{Z}}\Hom_{k}(M^{j},N^{j+i}) $ consisting of morphisms $ f $ such that $ f(ma)=f(m)a $ for all $ m $ in $ M $ and all $ a $ in $ A $, together with the differential $ d $ given by
	$$ d(f)=f\circ d_{M}-(-1)^{|f|}d_{N}\circ f $$ for a homogeneous morphism $ f $ of degree $ |f| $.
	
	The category $ \cc(A) $ of dg $ A $-modules is the category whose objects are the right dg $ A $-modules, and whose morphisms are the 0-cycles of the morphism complexes. It is an abelian category and a Frobenius category for the conflations which are split exact as sequences of graded $ A $-modules (see \cite{kellerDerivingDGCategories1994a}). Its stable category $ \ch(A) $ is called the \emph{homotopy category} of right dg $ A $-modules, which is equivalently defined as the category whose objects are the dg $ A $-modules and whose morphism spaces are the 0-th homology groups of the morphism complexes.
	
	The homotopy category $ \ch(A) $ is a triangulated category whose suspension functor $ \Si $ is the shift of dg modules $ M\mapsto \Si M $. The \emph{derived category} $ \cd(A) $ of dg $ A $-modules is the localization of $ \ch(A) $ at the full subcategory of acyclic dg $ A $-modules. 
	
	A dg $ A $-module $ P $ is \emph{cofibrant} if, for every surjective quasi-isomorphism $ L\ra M $, every morphism $ P\ra M $ factors through $ L $. For example, the dg algebra $ A $ itself is cofibrant. A dg $ A $-module $ I $ is \emph{fibrant} if, for every injective quasi-isomorphism $ L\ra M $, every morphism $ L\ra I $ extends to $ M $.
	
	\begin{Prop}\cite{kellerDerivingDGCategories1994a}
		\begin{itemize}
			\item[a)] For each dg $ A $-module $ M $, there is a quasi-isomorphism $ \mathbf{p}M\ra M $ with cofibrant $ \mathbf{p}M $ and a quasi-isomorphism $ M\ra \mathbf{i}M $ with fibrant $ \mathbf{i}M $.
			\item[b)] The projection functor $ \ch(A)\ra \cd(A) $ admits a fully faithful left adjoint given by $ M\mapsto\mathbf{p}M\ $ and a fully faithful right adjoint given by $ M\mapsto\mathbf{i}M $.
		\end{itemize}
	\end{Prop}
	We call $ \mathbf{p}M\ra M $ a \emph{cofibrant resolution} of $ M $ and $ M\ra \mathbf{i}M $ a \emph{fibrant resolution} of $ M $. We can compute morphisms in $ \cd(A) $ via
	$$ \cd(A)(L,M)\simeq\ch(A)(\mathbf{p}L,M)=H^{0}(\ch om_{A}(\mathbf{p}L,M)) .$$
	The \emph{perfect derived category} $ \per A $ is the smallest full subcategory of $ \cd(A) $ containing $ A $ which is stable under taking shifts, extensions and direct summands. An object $ M $ of $ \cd(A) $ belongs to $ \per(A) $ if and only if it is \emph{compact}, i.e. the functor $ \Hom_{\cd(A)}(M,?) $ commutes with arbitrary (set-indexed) direct sums (see~\cite[Section 5]{kellerDerivingDGCategories1994a}). The \emph{perfectly valued derived category} $ \pvd(A) $ is the full subcategory of $ \cd(A) $ consisting of those dg $ A $-modules whose underlying dg $ k $-module is perfect.
	Let $ f:B\ra A $ be a dg algebra morphism. Then $ f $ induces the restriction functor $ f_{*}: \cc(A)\ra\cc(B) $. It fits into the usual triple of adjoint functors $ (f^{*} , f_{*},f^{!}) $ between $ \cd(A) $ and $ \cd(B) $.

	\subsection{Derived categories of pseudocompact dg algebras}
	Let $ k $ be a field and $ R $ a finite dimensional separable $ k $-algebra (i.e. $ R $ is projective as a bimodule over itself). By definition, an $ R $-algebra is an algebra in the monoidal category of $ R $-bimodules.
	\begin{Def}\rm\cite[Chaper IV]{gabrielCategoriesAbeliennes1962} An $ R $-algebra $ A $ is \emph{pseudocompact} if it is endowed with a linear topology for which it is complete and separated and admits a basis of neighborhoods of zero formed by left ideals $ I $ such that $ A/I $ is finite dimensional over $ k $.	
	\end{Def}
	
	\begin{Rem}
		Recall~\cite[Lemma 4.1]{vandenberghCalabiYauAlgebrasSuperpotentials2015} that the notion of pseudocompact $ R $-algebra can be defined equivalently using right respectively two-side ideals.
	\end{Rem}

	Let $ A $ be a pseudocompact $ R $-algebra. A right $ A $-module $ M $ is \emph{pseudocompact} if it is endowed with a linear topology admitting a basis of neighborhoods of zero formed by submodules $ N $ of $ M $ such that $ M/N $ is finite-dimensional over $ k $. A \emph{morphism} between pseudocompact modules is a continuous $ A $-linear map. We denote by $ \mathrm{Pcm}(A) $ the category of pseudocompact right $ A $-modules. 
	
	The common annihilator of the simple pseudocompact $ A $-modules is called the \emph{radical} of $ A $ and is denoted by $ \mathrm{rad}A $.
	\begin{Prop}\cite[Proposition 4.3]{vandenberghCalabiYauAlgebrasSuperpotentials2015}
		The radical of $ A $ coincides with the ordinary Jacobson radical of $ A $.
	\end{Prop}
	
	%

	\begin{Def}\rm\cite[Appendix A]{kellerDerivedEquivalencesMutations2011} Let $ k $ and $ R $ be as above. Let $ A $ be an algebra in the category of graded $ R $-bimodules. We say that $ A $ is \emph{pseudocompact} if it is endowed with a complete separated topology admitting a basis of neighborhoods of 0 formed by graded ideals $ I $ of finite total codimension in $ A $. We denote by $ \mathrm{PcGr(R)} $ the category of pseudocompact graded $ R $-algebras. 
		
	We say that a dg $ R $-algebra $ A $ is \emph{pseudocompact} if it is pseudocompact as a graded $ R $-algebra and the differential $ d_{A} $ is continuous. Denote by $ \mathrm{PcDg(R)} $ the category of pseudocompact dg $ R $-algebras. We say that a morphism $ f:A\ra A' $ in $ \mathrm{PcDg(R)} $ is \emph{quasi-isomorphism} if it induces an isomorphism $ H^{*}(A)\ra H^{*}(A') $. Finally, we say that a dg $ R $-algebra is \emph{connective} if $ H^{p}(A) $ vanishes for all $ p>0 $.
		
	\end{Def}

	\begin{Def}\rm\label{Def:augmented}
		An \emph{augmented $ R $-algebra} is an algebra $ A $ equipped with $ k $-algebra homomorphisms $$ R\xrightarrow{\eta} A\xrightarrow{\epsilon} R $$ such that $ \epsilon\circ\eta $ is the identity. We use the notation $ \mathrm{PcAlg}(R) $ for the category of augmented pseudocompact dg $ R $-algebras. We say that a pseudocompact augmented dg $ R $-algebra $ A $ is \emph{complete} if $ \overline{A}:=\ker\epsilon=\mathrm{rad}A $. We denote by $ \mathrm{PcAlgc}(R) $ the full subcategory of $ \mathrm{PcAlg}(R) $ consisting of complete algebras.
	\end{Def}
	
	Let $ S $ be another finite dimensional separable $ k $-algebra and $ B $ a pseudocompact dg $ S $-algebra. A \emph{morphism} $ f $ from $ B $  to $ A $ consists of a $ k $-algebra morphism (not necessarily unital) $ f_{0}:S\ra R $ and a dg $ k $-algebra morphism (not necessarily unital)  $ f_{1}:B\ra A $ such that the following square commutes in the category of $ k $-algebras
	\begin{align*}
		\xymatrix{
			S\ar[d]_{\eta_{B}}\ar[r]^{f_{0}}&R\ar[d]^{\eta_{A}}\\
			B\ar[r]^{f_{1}}&A.
		}
	\end{align*}
	Similarly, we define morphisms from an object in $ \mathrm{PcAlg}(S) $ (respectively $ \mathrm{PcAlgc}(S) $) to an object in $ \PcAlg $ (respectively $ \PcAlgc $).

	\begin{Ex}\cite[Appendix A]{kellerDerivedEquivalencesMutations2011}
		Let $ Q $ be a finite graded quiver. We take $ R $ to be the product of copies of $ k $ indexed by the vertices of $ Q $ and $ A $ to be the completed path algebra, i.e.$ \, $for each integer $ n $, the component $ A^{n} $ is the product of the spaces $ kp $, where $ p $ ranges over the paths in $ Q $ of total degree $ n $. We endow $ A $ with a continuous differential sending each arrow to a possibly infinite linear combination of paths of length $ \geqslant2 $. For each $ n $, we define $ I_{n} $ to be the ideal generated by the paths of length $ \geq n $ and we define the topology on $ A $ to have the $ I_{n} $ as a basis of neighborhoods of $ 0 $. Then $ A $ is an augmented pseudocompact complete dg $ R $-algebra.
	\end{Ex}
	
	Let $ A $ be a pseudocompact dg $ R $-algebra. A right dg $ A $-module $ M $ is \emph{pseudocompact} if it is endowed with a topology for which it is complete and separated (in the category of graded $ A $-modules) and which admits a basis of neighborhoods of $ 0 $ formed by dg submodules of finite total codimension. It is clear that $ A $ is a pseudocompact dg module over itself. 
	
	A \emph{morphism} between pseudocompact dg $ A $-modules is a continuous dg $ A $-module morphism. We denote by $ \cc_{pc}(A) $ the the category of pseudocompact dg right $ A $-modules. A morphism $ f:L\ra M $ between pseudocompact dg right $ A $-modules is a \emph{quasi-isomorphism} if it induces an isomorphism $ H^{*}(L)\ra H^{*}(M) $.

	\begin{Prop}\cite[Lemma A.12]{kellerDerivedEquivalencesMutations2011}
		\begin{itemize}
			\item[a)] The homology $ H^{*}(A) $ is pseudocompact graded $ R $-algebra. In particular, $ H^{0}(A) $ is a pseudocompact R-algebra.
			\item[b)] For each pseudocompact dg module $ M $, the homology $ H^{*}(M) $ is a pseudocompact graded module over $ H^{*}(A) $.
		\end{itemize}	
	\end{Prop}
	
	The category $ \cc_{pc}(A) $ has a natural dg enrichment $ \cc_{pc}^{dg}(A) $. It has the same objects as $ \cc_{pc}(A) $. For $ L,M\in\cc_{pc}^{dg}(A) $, the \emph{Hom-complex} $ \ch om_{\cc_{pc}^{dg}(A)}(L,M) $ is the complex of $ R $-modules with degree $ p $ component $ \ch om_{\cc_{pc}^{dg}(A)}^{p}(L,M) $ being the space of continuous homogeneous $ A $-linear maps of degree $ p $ form $ L $ to $ M $ (here we consider $ A $ as a pseudocompact graded algebra and $ L,M $ as pseudocompact graded $ A $-modules) and with its differential defined by 
	
	$$ d(f)=d_{M}\circ f-(-1)^{p}f\circ d_{L} $$ for $ f\in\Hom_{\cc_{pc}^{dg}(A)}^{p}(L,M) $.
	
	For a dg category $ \ca $, 
	the category $ H^{0}(\ca) $ has the same objects as $ \ca $ and its morphisms are defined by $$(H^{0}(\ca))(X,Y)=H^{0}(\ca(X,Y)), $$ where $ H^{0} $ denotes the $ 0 $-th homology of the complex $ \ca(X,Y) $. We denote by $ \ch_{pc}(A) $ the zeroth homology category of $ \cc_{pc}^{dg}(A) $,$ \, $i.e.$ \, $$ \ch_{pc}(A)=H^{0}(\cc_{pc}^{dg}(A)) $.

	\begin{Thm}\cite[Section 8.2]{positselskiTwoKindsDerived2011}
		Let $ A $ be an object in $ \PcAlgc $. There exists a stable model structure on $ \cc_{pc}(A) $ which is given as follows:
		\begin{itemize}
			\item[(1)] The weak equivalences are the morphisms with an acyclic cone. Here, an object is acyclic if it is in the smallest subcategory of the homotopy category of A which contains the total complexes of short exact sequences and is closed under arbitrary products.
			\item[(2)] The cofibrations are the injective morphisms with cokernel which is projective when forgetting the differential.
			\item[(3)] The fibrations are the surjective morphisms.
		\end{itemize}
		Since this model structure is stable, the corresponding homotopy category becomes a triangulated category and we called it the \emph{pseudocompact derived category} of $ A $ and denote it by $ \cd_{pc}(A) $.
	\end{Thm}
	
	\begin{Rem}\cite[pp.10]{vandenberghCalabiYauAlgebrasSuperpotentials2015}
		Under suitable boundedness assumptions (algebras concentrated in degrees $ \leqslant0 $ and modules concentrated in degrees $ \leqslant N $), weak equivalence is the same as quasi-isomorphism.
	\end{Rem}
	Let $ A $ be an object in $ \PcAlgc $. We define the \emph{perfect derived category} $ \per_{pc}(A) $ to be the thick subcategory of $ \cd_{pc}(A) $ generated by the free $ A $-module of rank 1. The \emph{perfectly valued derived category} $ \pvd_{pc}(A) $ is defined to be the full subcategory of $ \cd_{pc}(A) $ whose objects are the pseudocompact dg modules $ M $ such that $ \Hom_{\cd_{pc}(A)}(P, M) $ is finite-dimensional for each perfect $ P $.
	
	
	Let $ A' $ be another pseudocompact dg $ R $-algebra. Their \emph{complete tensor product} is defined by
	$$ A\widehat{\otimes}_{k}A'=\varprojlim_{U,V}A/U\ten_{k}A'/V, $$ where $ U,V $ run through the system of open neighborhoods of zero in $ A $ and $ A' $ respectively. Then $ A\widehat{\otimes}_{k}A'$ is also pseudocompact. We define the \emph{enveloping algebra} $ A^{e} $ of $ A $ to be the complete tensor product $ A\widehat{\otimes}_{k}A^{op} $. The \emph{category of pseudocompact $ A$-$A' $-bimodules} is defined as $ \cc_{pc}(A\widehat{\otimes}_{k}A') $.
	
	The dg algebra $ A $ is a pseudocompact dg module over the enveloping algebra $ A\widehat{\otimes}_{k}A^{op} $. We say that $ A $ is \emph{(topologically homologically) smooth} if the module $ A $ considered as a pseudocompact dg module over $ A^{e}$ lies in $ \per_{pc}(A^{e}) $.
	
	\begin{Prop}\cite[Lemma A.13]{kellerDerivedEquivalencesMutations2011}
		If $ A $ is the completed path algebra of a finite graded quiver endowed with a continuous differential sending each arrow to a possibly infinite linear combination of paths of length $ \geq2 $, then $ A $ is smooth.
		
	\end{Prop}
	%
	%
	%
	
	
	\begin{Prop}\cite[Proposition A.14]{kellerDerivedEquivalencesMutations2011}\label{Prop:pseudocompact for smooth dg}
		Let $ A $ be a pseudocompact dg $ R $-algebra. Assume that $ A $ is smooth and connective.
		\begin{itemize}
			\item[a)] The canonical functor $ \ch_{pc}(A)\ra\cd_{pc}(A) $ has a left adjoint $ M\mapsto\mathbf{p}M $.
			\item[b)] The triangulated category $ \pvd_{pc}A $ is generated by the dg modules of finite dimension concentrated in degree 0.
			\item[c)] The full subcategory $ \pvd_{pc}A $ of $ \cd_{pc}(A) $ is contained in the perfect derived category $ \per_{pc}A $.
			\item[d)] The opposite category $ \cd_{pc}(A)^{op} $ is compactly generated by $ \pvd_{pc}A $.
			\item[e)] Let $ A\ra A' $ be a quasi-isomorphism of pseudocompact, connective, smooth dg algebras. Then the restriction functor $ \cd_{pc}(A')\ra\cd_{pc}(A) $ is an equivalence. In particular, if the homology of $ A $ is concentrated in degree 0, there is an equivalence $ \cd_{pc}(A)\ra\cd_{pc}(H^{0}(A)) $. Moreover, in this case $ \cd_{pc}(H^{0}A) $ is equivalent to the derived category of the abelian category $ \mathrm{Pcm}(H^{0}A) $.
			\item[f)] Assume that $ A $ is a complete dg path algebra. There is an equivalence between $ \cd_{pc}(A)^{op} $ and the localizing subcategory $ \cd_{0}(A) $ of the ordinary derived category $ \cd(A) $ generated by the finite-dimensional dg $ A $-modules.
		\end{itemize}
	\end{Prop}
	\begin{Rem}
		In $ a) $ of the above Proposition, we only need the condition ‘smooth’. In $ b) $, we need the condition ‘connective’.
	\end{Rem}
	
	%
	
	For two objects $ L $ and $ M $ of $ \cd_{pc}(A) $, define
	$$ \RHom_{A}(L,M)=\ch om_{\cc^{dg}_{pc}(A)}(\mathbf{p}L,M) .$$
	Then we have $ \Hom_{\cd_{pc}(A)}(L,M)=\ch_{pc}(A)(\mathbf{p}L,M)=H^{0}(\ch om_{\cc^{dg}_{pc}(A)}(\mathbf{p}L,M)).$
	
	Let $ Y $ be a pseudocompact dg $ A $-$ A' $-bimodule and $ X $ a pseudocompact dg right $ A $-module. Let $ \mathbf{p}X $ be a cofibrant replacement of $ X $ in $ \cc_{pc}(A) $. Their \emph{complete derived tensor product} $ X\widehat{\otimes}^{\mathbf{L}}_{A}Y $ is defined by
	$$ X\widehat{\otimes}^{\mathbf{L}}_{A}Y=\varprojlim_{U,V}\mathbf{p}X/U\ten_{A}Y/V, $$ where $ U,V $ run through the system of open neighborhoods of zero in $ \mathbf{p}X $ and $ Y $ respectively. Up to quasi-isomorphism, it is well defined.

	Moreover, we have the tensor-Hom adjunction
	$$ \RHom_{A'}(X\widehat{\ten}^{\bL}_{A}Y,Z)\simeq\RHom_{A}(X,\RHom_{A'}(Y,Z)) $$ for objects $ X\in\cd_{pc}(A) $ and $ Z\in\cd_{pc}(A') $.
	
	\begin{Def}\em
		Let $ A $ be an object in $ \PcAlgc $. For a pseudocompact dg $ A $-bimodule $ M $, we define its \emph{derived dual} $ M^{\vee} $ as
		$$ M^{\vee}=\RHom_{A^{e}}(M,A^{e}) .$$ 
		In particular, the \emph{inverse dualizing bimodule} of $ A $ is defined as $ A^{\vee} $. 
		
		Let $ S $ be another finite dimensional separable $ k $-algebra and $ B $ an object in $ \mathrm{PcAlgc}(S) $. Let $ f $ be a morphism from $ B $ to $ A $. The \emph{inverse dualizing bimodule} $ \Theta_{f} $ of $ f $ is defined as
		$$ \Theta_{f}=\RHom_{A^{e}}(\cone(A\lcten_{B}A\ra A),A^{e}) ,$$ where $ A\lcten_{B}A $ is isomorphic to $ A\lcten_{B}B\lcten_{B}A $. 
	\end{Def}
	
	The morphism $ f $ induces the restriction functor $ f_{*}: \cc_{pc}(A)\ra\cc_{pc}(B) $. It fits into the usual adjunction pair $ (\bL f^{*}=?\lcten_{B}A , f_{*}) $ between $ \cd_{pc}(A) $ and $ \cd_{pc}(B) $. The morphism $ f $ also induces a morphism $ f^{e}:A^{e}\ra B^{e} $ between their enveloping algebra. By abuse of notation, we also denote the corresponding adjoint functors between $ \cd_{pc}(A^{e}) $ and $ \cd_{pc}(B^{e}) $ by $ (\bL f^{*}, f_{*}) $. For a pseudocompact dg $ B $-bimodule $ M $, we have the following useful formula
	\begin{equation*}
		\begin{split}
			\bL f^{*}(M)=&M\lcten_{B^{e}}A^{e}\\
			\cong&A\lcten_{B}M\lcten_{B}A.
		\end{split}
	\end{equation*}
	In particular, if we take $ M=B $, then $ \bL f^{*}(\cb)\cong A\lcten_{B}A $.

	\begin{Prop}\cite{kellerDerivedEquivalencesMutations2011,adachiDiscretenessSiltingObjects2019}
		Let $ A $ be a pseudocompact dg $ R $-algebra. The forgetful functor $ \cd_{pc}(A)\ra\cd(A) $ restricts to a triangle equivalence $ \per_{pc}(A)\ra\per(A) $. If $ A $ is also smooth, then it restricts to a triangle equivalence $ \pvd_{pc}(A)\ra\pvd(A) $.
	\end{Prop}
	
	\begin{Cor}
		Let $ f:A\ra A' $ be a morphism in $ \PcAlg $. Assume that $ A $ and $ A' $ are connective and smooth. Then $ f_{*}:\cd(A')\ra\cd(A) $ is a triangle equivalence if and only if $ f_{*}:\cd_{pc}(A')\ra\cd_{pc}(A) $ is.
	\end{Cor}
	
	\subsection{Hochschild/cyclic homology in the pseudocompact setting}

	Let $ \Lambda $ be the dg algebra generated by an indeterminate $ \epsilon $ of cohomological degree $ -1 $ with $ \epsilon^{2}=0 $ and $ d\epsilon=0 $. The underlying complex of $ \Lambda $ is
	$$ \cdots\rightarrow k\epsilon\rightarrow k\rightarrow0\ra\cdots. $$
	Then a \emph{mixed complex} over $ k $ is a dg right $ \Lambda $-module whose underlying dg $ k $-module is $ (M,b) $ and where $ \epsilon $ acts by a closed endomorphism $ B $. Suppose that $ M=(M,b,B) $ is a mixed complex. Then the \emph{shifted mixed complex} $ \Si M $ is the mixed complex such that $ (\Si M)^{p}=M^{p-1} $ for all $ p $, $ b_{\Si M}=-b $ and $ B_{\Si M}=-B $. Let $ f:M\rightarrow M' $ be a morphism of mixed complexes. Then the \emph{mapping cone} over $ f $ is the mixed complex
	$$ \bigg(M'\oplus M,\begin{bmatrix}
		b_{M'} & f \\
		0 & -b_{M}
	\end{bmatrix},\begin{bmatrix}
		B_{M'} & 0 \\
		0 & -B_{M}
	\end{bmatrix}\bigg). $$
	We denote by $ \cm ix $ the category of mixed complexes and by $ \cd\cm ix $ the derived category of the dg algebra $ \Lambda $.

	Let $ A $ be an object in $ \mathrm{PcAlgc}(R) $ (see Definition~\ref{Def:augmented}). For an $ R $-bimodule $ U $, we define $ U_{R}=U/[R,U] $ and we let $ U^{R} $ be the $ R $-centralizer in $ U $. We associate a precyclic chain complex $ C(A) $ (see~\cite[Definition 2.5.1]{lodayCyclicHomology2013}) with $ A $ as follows: For each $ n\in\mathbb{N} $, its $ n $-th term is
	$$ C_{p}(A)=(A\ten_{R}\overline{A}^{\ten_{R}^{p}})_{R} .$$
	The degeneracy maps are given by
	\begin{equation*}
		d_{i}(a_{n},\ldots,a_{i},a_{i-1},\ldots,a_{0})=\left\{
		\begin{aligned}
			(a_{n},\ldots,a_{i}a_{i-1},\ldots,a_{0})&&\text{if $ i>0 $,}\\
			(-1)^{n+\sigma}(a_{0}a_{n},\ldots,a_{1})&&\text{if $ i=0 $},
		\end{aligned}
		\right.
	\end{equation*}
	where $ \sigma=(\mathrm{deg}a_{0})(\mathrm{deg}a_{1}+\cdots+\mathrm{deg}a_{n-1}) $. The cyclic operator is given by
	$$ t(a_{n-1},\ldots,a_{0})=(-1)^{n+\sigma}(a_{0},a_{n-1},a_{n-2},\cdots,a_{1}) .$$

	Then the corresponding \emph{product total complex} $ (H\!H(A),b) $ of $ (C(A),b=\sum_{i=0}^{n}(-1)^{i}d_{i}) $ is called the \emph{normalized Hochschild complex} of $ A $. The \emph{(continuous) Hochschild homology} of $ A $ is defined to be the cohomology of this complex. By~\cite[Proposition B.1]{vandenberghCalabiYauAlgebrasSuperpotentials2015}, the normalized Hochschild complex is quasi-isomorphic to $ A\lcten_{A^{e}}A $ in $ \cd(k) $.

	We associate a mixed complex $ (M(A),b,B) $ with this precyclic chain complex as follows: Consider the \emph{product total complex} $ (H\!H(A),b’) $ of $ (C(A),b’=\sum_{i=0}^{n-1}(-1)^{i}d_{i}) $. The underlying dg module of $ M(A) $ is the mapping cone over $ (1 − t) $ viewed as a morphism of complexes
	$$ 1-t:(H\!H(A),b')\ra(H\!H(A),b) ,$$ where $ b=\sum_{i=0}^{n}(-1)^{i}d_{i} $ and $ b'=\sum_{i=0}^{n-1}(-1)^{i}d_{i} $.
	Its underlying module is $ H\!H(A)\oplus H\!H(A) $; it is endowed with the grading whose $ n $-th component is $ H\!H(A)_{n}\oplus H\!H(A)_{n-1} $ and the differential is 
	$$ \begin{bmatrix} b & 1-t \\ 0 & -b' \end{bmatrix} .$$
	The operator $ B:M\ra M $ is
	$$ \begin{bmatrix} 0 & 0 \\ N & 0 \end{bmatrix} ,$$ where $ N=\sum_{i=0}^{n}t^{i} $.

	Let $ S $ be an other finite dimensional separable $ k $-algebra and $ B $ a pseudocompact dg $ S $-algebra. Let $ f $ be a morphism from $ B $ to $ A $. Then $ f $ induces a canonical morphism between their Hochschild complexes
	$$ \gamma_{f}:H\!H(B)\ra H\!H(A) .$$
	\begin{Def}\rm
		The \emph{(continuous) Hochschild homology} $ H\!H_{\bullet}(f) $ of $ f $ is the cohomology of the \emph{relative Hochschild complex} which is defined as follows
		$$ H\!H(f)=\cone(\gamma_{f}:HH(B)\ra HH(A)) .$$
	\end{Def}

\begin{Rem}
	By definition, a closed element $ \xi=(s\xi_{B},\xi_{A})\in \cone(\gamma_{f}:HH(B)\ra HH(A)) $ of degree $ -n $ consists of an element $ \xi_{B} \in HH(B) $ of degree $ -n+1 $, together with an element $ \xi_{A}\in HH(A) $ of degree $ -n $, such that $ b_{B}(\xi_{B})=0 $ and $ b_{A}(\xi_{A})+\gamma_{f}(\xi_{B})=0 $, where $ b_{B} $ and $ b_{A} $ are the Hochschild differentials of $ B $ and $ A $ respectively.
\end{Rem}
	
	\begin{Def}\rm
		The \emph{(continuous) cyclic homology} $ HC_{\bullet}(A) $ of $ A $ is defined to be the cohomology of the \emph{cyclic chain complex} of $ A $
		$$ HC(A)=M(A)\lten_{\Lambda}k .$$
		
		The \emph{(continuous) negative cyclic homology} $ HN_{\bullet}(A) $ of $ A $ is defined to be the cohomology of the \emph{negative cyclic chain complex} of $ A $
		$$ HN(A)=\RHom_{\Lambda}(k,M(A)). $$
	\end{Def}
	
	The augmentation morphism $ \Lambda\to k $ induces natural morphisms in $ \cd(k) $
	$$ HN(A)\rightarrow HH(A)\rightarrow HC(A).$$
	
	The morphism $ f $ also induces a canonical morphism between their mixed complexes
	$$ \gamma_{f}:M(B)\ra M(A) .$$
	We denote by $ M(f) $ the mapping cone over $ \gamma_{f} $.
	\begin{Def}\rm
		The \emph{(continuous) cyclic homology} $ HC_{\bullet}(f) $ of $ f:B\rightarrow A $ is defined to be the cohomology of the \emph{cyclic chain complex group} of $ f $
		$$ HC(f)=M(f)\lten_{\Lambda}k. $$
		The \emph{(continuous) negative cyclic homology} $ HN_{\bullet}(f) $ of $ f:B\rightarrow A $ is defined to be the cohomology of the \emph{negative cyclic chain complex} of $ f $
		$$ HN(f)=\RHom_{\Lambda}(k,M(f)). $$
	\end{Def}

	Similarly, the augmentation morphism $ \Lambda\to k $ induces natural morphisms in $ \cd(k) $
	$$ HN(f)\rightarrow HH(f)\rightarrow HC(f).$$

	\section{Relative Calabi--Yau completions in the pseudocompact setting}\label{Section3}
	Let $ k $ be a field and $ R $ a finite dimensional separable $ k $-algebra. Let $ A $ be an object in $ \mathrm{PcAlgc}(R) $. 
	\begin{Def}\rm
		Let $ M $ be a pseudocompact dg $ A $-bimodule. The \emph{completed tensor algebra} $ T_{A}(M) $ is defined as
		$$ T_{A}(M)={\displaystyle \prod_{n=0}^{\infty}M^{\widehat{\otimes}^{n}}}, $$ where $ M^{\widehat{\otimes}^{0}}=A $ and $ M^{\widehat{\otimes}^{n}}=\underbrace{M\widehat{\otimes}_{A}M\widehat{\otimes}_{A}\cdots\widehat{\otimes}_{A}M}_{\text{$ n $-times}} $ for $ n\geqslant1 $. The dg algebra structure on $ T_{A}(M) $ is given by the differentials of $ A $ and $ M $ and the multiplication is given by the concatenation product.
		
		The \emph{derived completed tensor algebra} is defined as
		$$ \mathbf{L}T_{A}(M)=T_{A}(\mathbf{p}M) $$
		where $ \mathbf{p}M $ is a cofibrant replacement of $ M $ as a pseudocompact dg $ A $-bimodule. Up to weak equivalence, it is well defined.
	\end{Def}
	The ideals $ T_{A}(M)_{\geqslant s}=\prod_{n\geqslant s}M^{\widehat{\otimes}^{n}} $ are clearly closed in $ T_{A}(M) $ and we have 
	$$ T_{A}(M)=\varprojlim_{s\in\mathbb{N}}T_{A}(M)/T_{A}(M)_{\geqslant s} .$$
	Thus, the completed tensor algebra is again in $ \mathrm{PcAlgc}(R) $.

	\subsection{Relative deformed Calabi--Yau completions}
	Let $ S $ be another finite dimensional separable $ k $-algebra and $ B $ an object in $ \mathrm{PcAlgc}(S) $. Let $ f $ be a morphism from $ B $ to $ A $. We assume that $ B $ and $ A $ are topologically homologically smooth and connective. Let $ [\xi] $ be an element in $ H\!H_{n-2}(f) $. Our objective is to define the deformed relative $ n $-Calabi–Yau completion of $ f:B\ra A $ with respect to the Hochschild homology class $ [\xi]\in H\!H_{n-2}(f) $.
	
	The morphism $ f $ induces a morphism in $ \cd_{pc}(A^{e}) $
	$$ m_{f}:B\lcten_{B^{e}}A^{e}\ra A .$$
	
	After taking the bimodule dual, using the smoothness of $ B $, we get a morphism
	$$ m_{f}^{\vee}:A^{\vee}\ra B^{\vee}\lcten_{B^{e}}A^{e}. $$ Let $ \Xi $ be the cofiber of $ m_{f}^{\vee} $.
	The dualizing bimodule $ \Theta_{f}=(\cof(B\lcten_{B^{e}}A^{e}\ra A))^{\vee} $ of $ f $ is quasi-isomorphic to $ \Si^{-1}\Xi $. 
	
	By the definition of Hochschild homology of $ f $, we have the following long exact sequence
	$$ \cdots\ra H\!H_{n-2}(B)\ra H\!H_{n-2}(A)\ra H\!H_{n-2}(f)\ra H\!H_{n-3}(B)\ra\cdots. $$ Thus, the Hochschild homology class $ [\xi]\in H\!H_{n-2}(f) $ induces an element $ [\xi_{B}] $ in $ H\!H_{n-3}(B) $.
	
	Notice that since $ B,A $ are smooth, we have the following isomorphisms:
	\begin{equation*}
		\begin{split}
			\Hom_{\cd_{pc}(B^{e})}(\Si^{n-2}B^{\vee},\Si B)\simeq& H^{3-n}(B\lcten_{B^{e}}B)=HH_{n-3}(B),\\
			\Hom_{\cd_{pc}(A^{e})}(\Si^{n-2}\Xi,\Si A)\simeq& H^{2-n}(\cone(B\lcten_{B^{e}}A\ra A\lcten_{A^{e}}A))\\
			\leftarrow&H^{2-n}(\cone(B\lcten_{B^{e}}B\ra A\lcten_{A^{e}}A))\\
			\simeq&HH_{n-2}(f).
		\end{split}
	\end{equation*}
	
	Thus, the homology class $ [\xi] $ induces a morphism in $ \cd_{pc}(A^{e}) $
	$$ \xi:\Si^{n-2}\Xi\ra\Si A $$ and the homology class $ [\xi_{B}] $ induces a morphism in $ \cd_{pc}(B^{e}) $
	$$ \xi_{B}:\Si^{n-2}B^{\vee}\ra\Si B .$$
	Moreover, we have the following commutative diagram in $ \cd_{pc}(A^{e}) $
	\[
	\begin{tikzcd}
		\bL f^{*}(\Si^{n-1}B^{\vee})\arrow[r]\arrow[d,"\xi_{B}"]&\Si^{n-2}\Xi\arrow[d,"\xi"]\\
		\bL f^{*}(\Si B)\arrow[r]&\Si A.
	\end{tikzcd}
	\]
	Therefore, the morphism $ \xi_{B} $ gives rise to a ‘deformation’
	$$ \bm\Pi_{n-1}(B,\xi_{B}) $$ of $ \bm\Pi_{n-1}(B)=\bL T_{B}(\Xi^{n-2}B^{\vee}) $, obtained by adding $ \xi_{B} $ to the differential of $ \bm\Pi_{n-1}(B) $; the morphism $ \xi $ gives rise to a ‘deformation’
	$$ \bm\Pi_{n}(A,B,\xi) $$ of $ \bm\Pi_{n}(A,B)=\bL T_{A}(\Si^{n-2}\Xi) $, obtained by adding $ \xi $ to the differential of $ \bL T_{A}(\Si^{n-2}\Xi) $; and the commutative diagram above gives rise to a morphism
	$$ \tilde{f}:\bm\Pi_{n-1}(B,\xi_{B})\ra\bm\Pi_{n}(A,B,\xi) .$$
	A standard argument shows that up to weak equivalence, the morphism $ \tilde{f} $ and the deformations $ \bm\Pi_{n-1}(B,\xi_{B}) $, $ \bm\Pi_{n}(A,B,\xi) $ only depend on the class $ [\xi] $.
	
	\begin{Def}\rm\cite[Definition 3.14]{yeungRelativeCalabiYauCompletions2016}\label{Def:deformed CY completion}
		The dg functor $ \tilde{f} $ defined above is called the \emph{deformed relative n-Calabi–Yau completion} of $ f:B\ra A $ with respect to the Hochschild homology class $ [\xi]\in H\!H_{n-2}(f) $.
	\end{Def}
	
	
	\begin{Thm}\cite[Theorem 3.23]{yeungRelativeCalabiYauCompletions2016}\cite[Proposition 5.28]{bozecRelativeCriticalLoci2020}\label{Thm:Relative CY completion to left steucture}\label{Thm:Relative defomed CY completion has a canonical left CY}
		If $ [\xi] $ has a negative cyclic lift, then each choice of such a lift gives rise to a canonical left $ n $-Calabi–Yau structure on the morphism
		$$ \tilde{f}:\bm\Pi_{n-1}(B,\xi_{B})\ra\bm\Pi_{n}(A,B,\xi) .$$
	\end{Thm}
	
	\begin{Rem}
		In the pseudocompact setting, the situation is different from \cite[Theorem 3.23]{yeungRelativeCalabiYauCompletions2016} and \cite[Proposition 5.28]{bozecRelativeCriticalLoci2020}. But the proof for the above Theorem can be adapted from their proof.
	\end{Rem}

	%
	%
	
	\section{Ice quiver mutations and complete relative Ginzburg algebras}\label{Section4}

	\subsection{Ice quivers}
	\begin{Def}\rm
		A \emph{quiver} is a tuple $ Q=(Q_{0}, Q_{1},s,t) $, where $ Q_{0} $ and $ Q_{1} $ are sets, and $ s, t:Q_{1}\ra Q_{0} $ are functions. We think of the elements of $ Q_{0} $ as vertices and those of $ Q_{1} $ as arrows, so that each $ \alpha\in Q_{1} $ is realised as an arrow $ \alpha:s(\alpha)\ra t(\alpha) $. We call $ Q $ \emph{finite} if $ Q_{0} $ and $ Q_{1} $ are finite sets.
	\end{Def}
	\begin{Def}\rm
		Let $ Q $ be a quiver. A quiver $ F=(F_{0}, F_{1}, s', t') $ is a \emph{subquiver} of $ Q $ if it is a quiver such that $ F_{0}\subseteq Q_{0} $, $ F_{1}\subseteq Q_{1} $ and the functions $ s' $ and $ t' $ are the restrictions of $ s $ and $ t $ to $ F_{1} $ . We say $ F $ is a \emph{full subquiver} if $ F_{1}=\{\alpha\in Q_{1}\,:s(\alpha),t(\alpha)\in F_{0} \} $, so that a full subquiver of $ Q $ is completely determined by its set of vertices.	
	\end{Def}
	\begin{Def}\rm
		An \emph{ice quiver} is a pair $ (Q,F) $, where $ Q $ is a finite quiver, $ F $ is a (not necessarily full) subquiver of $ Q $.
		We call $ F_{0} $, $ F_{1} $ and $ F $ the frozen vertices, arrows and subquiver respectively. We also call $ Q_{0}\setminus F_{0} $ and $ Q_{1}\setminus F_{1} $ the unfrozen vertices and arrows respectively.	
	\end{Def}
	\subsection{Combinatorial mutations}
	\begin{Def}\rm(Pressland~\cite[Deﬁnition 4.4]{presslandMutationFrozenJacobian2020})\label{Def:Combinatorial mutations}
		Let $ (Q, F) $ be an ice quiver and let $ v $ be an unfrozen vertex such that no loops or 2-cycles of $ Q $ are incident with $ v $. Then the \emph{extended mutation} $ \mu_{v}^{P}(Q,F)=(\mu_{v}^{P}(Q),\mu_{2}^{P}(F)) $ of $ (Q, F) $ at $ v $ is defined to be the output of the following procedure.
		\begin{itemize}
			\item[(1)] For each pair of arrows $ \alpha:u\ra v $ and $ \beta:v\ra w $, add an unfrozen arrow $ [\beta\alpha]:u\ra w $ to $ Q $.
			\item[(2)] Replace each arrow $ \alpha:u\ra v $ by an arrow $ \alpha^{*}:v\ra u $, and each arrow $ \beta:v\ra w $ by an arrow $ \beta^{*}:w\ra v $.
			\item[(3)] Remove a maximal collection of unfrozen 2-cycles, i.e. 2-cycles avoiding the subquiver $ F $.
			\item[(4)] Choose a maximal collection of half-frozen 2-cycles, i.e. 2-cycles in which precisely one arrow is frozen. Replace each 2-cycle in this collection by a frozen arrow, in the direction of the unfrozen arrow in the 2-cycle.
		\end{itemize}
	\end{Def}
	
	\begin{Rem}
		Because of the choices involved in steps ($ 3) $ and $ (4) $, this operation is only defined up to quiver isomorphism. If we ignore all arrows between frozen vertices, we obtain the usual definition of Fomin--Zelevinsky mutation (see~\cite[Definition 4.2]{fominClusterAlgebrasFoundations2002}).
	\end{Rem}
	
	\begin{Ex}
		Consider the ice quiver $ (Q,F) $ given by
		
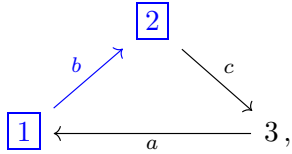
\[
		\begin{tikzcd}
			&\color{blue}\boxed{2}\arrow[dr,"c"]&\\
			\color{blue}\boxed{1}\arrow[ur,blue,"b"]&&3\arrow[ll,"a"]\,,
		\end{tikzcd}
		\]
		where the blue part is the subquiver $ F $.
		
		We perform the extended Fomin--Zelevinsky mutation at the vertex 3. Then we get the following ice quiver $ \mu_{3}^{P}(Q,F)=(Q',F') $
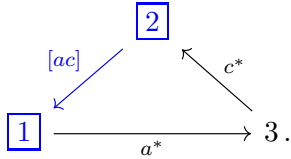
		\[
		\begin{tikzcd}
			&\color{blue}\boxed{2}\arrow[dl,blue,swap,"{[ac]}"]&\\
			\color{blue}\boxed{1}\arrow[rr,swap,"a^{*}"]&&3\arrow[ul,swap,"c^{*}"]\,.
		\end{tikzcd}
		\]
		
		If we preform the usual Fomin--Zelevinsky mutation at the vertex 3, we get a different final quiver
		
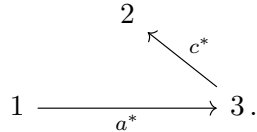
\[
		\begin{tikzcd}
			&2&\\
			1\arrow[rr,swap,"a^{*}"]&&3\arrow[ul,swap,"c^{*}"]\,.
		\end{tikzcd}
		\]	
	\end{Ex}

	\subsection{Ice quivers with potential}
	Let $ k $ be a field. Let $ Q $ be a finite quiver.
	\begin{Def}\rm\label{Def: semisimple algebra of vertices}
		Let $ S $ be the semisimple $ k $-algebra $ \prod_{i\in Q_{0}}ke_{i} $. The vector space $ kQ_{1} $ naturally becomes an $ S $-bimodule. Then the $ \emph{complete path algebra} $ of $ Q $ is the completed tensor algebra
		$$ \widehat{kQ}=T_{S}(kQ_{1}) .$$ It has underlying vector space 
		$$ \prod_{d=0}^{\infty}(kQ_{1})^{\otimes_{S}^{d}} ,$$ and multiplication given by concatenation. The algebra $ \widehat{kQ} $ becomes a graded pseudocompact $ S $-algebra.
	\end{Def}
	
	\begin{Rem}
		The complete path algebra is a pseudocompact algebra by equipping it with the $ \mathfrak{m} $-adic topology, where $ \mathfrak{m} $ is the two-sided \emph{arrow ideal}
		$$ \mathfrak{m}=\prod_{d=1}^{\infty}(kQ_{1})^{\otimes_{S}^{d}} .$$ 
	
		The closure of the ideal generated by a set $ R\subseteq \widehat{kQ} $ is
		$$ \overline{\langle R\rangle}=\{\sum_{i=1}^{\infty}a_{i}r_{i}b_{i}:a_{i},b_{i}\in\widehat{kQ},r_{i}\in R\}.$$
	\end{Rem}
	
	\begin{Def}\rm\cite[Deﬁnition 2.8]{presslandMutationFrozenJacobian2020}
		The natural grading on $ \widehat{kQ} $ induces a grading on the continuous Hochschild homology $ H\!H_{0}(\widehat{kQ})=\widehat{kQ}/[\widehat{kQ},\widehat{kQ}] $. A \emph{potential} on $ Q $ is an element $ W $ in $ H\!H_{0}(\widehat{kQ}) $ expressible as a (possibly infinite) linear combination of homogeneous elements of degree at least 2, such that any term involving a loop has degree at least 3. An \emph{ice quiver with potential} is a tuple $ (Q, F, W) $ in which $ (Q, F) $ is a finite ice quiver and $ W $ is a potential on $ Q $. If $ F=\emptyset $ is the empty quiver, then $ (Q,\emptyset,W)=(Q,W) $ is called simply a \emph{quiver with potential}. We say that $ W $ is \emph{irredundant} if each term of $ W $ includes at least one unfrozen arrow.
	\end{Def}
	
	A potential can be thought of as a formal linear combination of cyclic paths in $ Q $ (of length at least 2), considered up to the equivalence relation on such cycles induced by
	$$ \alpha_{n}\cdots\alpha_{1}\sim\alpha_{n-1}\cdots\alpha_{1}\alpha_{n} .$$
	
	\begin{Def}\rm
		Let $ p=\alpha_{n}\cdots\alpha_{1} $ be a cyclic path, with each $ \alpha_{i}\in Q_{1} $, and let $ \alpha\in Q_{1} $ be any arrow. Then the \emph{cyclic derivative} of $ p $ with respect to $ \alpha $ is
		$$ \partial_{\alpha}p=\Sigma_{\alpha_{i}=\alpha}\alpha_{i-1}\cdots\alpha_{1}\cdots\alpha_{i+1} .$$ We extend $ \partial_{\alpha} $ by linearity and continuity. Then it determines a map $ H\!H_{0}(\widehat{kQ})\ra\widehat{kQ} $. For an ice quiver with potential $ (Q,F,W) $, we define the \emph{relative Jacobian algebra}
		$$ J(Q,F,W)=\widehat{kQ}/\overline{\langle\partial_{\alpha}W:\alpha\in Q_{1}\setminus F_{1}\rangle} .$$
		If $ F=\emptyset $, we call $ J(Q,W)=J(Q,\emptyset,W) $ the \emph{Jacobian algebra} of the quiver with potential $ (Q, W) $.
	\end{Def}

	\begin{Def}\rm
		Let $ Q $ be a quiver. An ideal of $ \widehat{kQ} $ is called \emph{admissible} if it is contained in the square of the closed ideal generated by the arrows of $ Q $. We call an ice quiver with potential $ (Q,F,W) $ \emph{reduced} if $ W $ is irredundant and the Jacobian ideal of $ \widehat{kQ} $ determined by $ F $ and $ W $ is admissible. An ice quiver with potential $ (Q, F, W) $ is \emph{trivial} if its relative Jacobian algebra $ J(Q,F,W) $ is a product of copies of the base field $ k $.
	\end{Def}
	
	\begin{Def}\rm\cite[Deﬁnition 3.7]{presslandMutationFrozenJacobian2020}
		Let $ (Q,F,W) $ and$  (Q',F',W') $ be ice quivers with potential such that $ Q_{0}=Q'_{0} $ and $ F_{0}=F'_{0} $. In particular, this means that $ \widehat{kQ} $ and $ \widehat{kQ'} $ are complete tensor algebras over the same semisimple algebra $ S=kQ_{0} $. An isomorphism $ \varphi:\widehat{kQ}\ra\widehat{kQ'} $ is said to be a \emph{right equivalence} of the ice quivers with potential if
		\begin{itemize}
			\item[(1)] $ \varphi|_{S}=\id_{S} $,
			\item[(2)] $ \varphi(\widehat{kF})=\widehat{kF'} $, where $ \widehat{kF} $ and $ \widehat{kF'} $ are treated in the the natural way as subalgebras of $ \widehat{kQ} $ and $ \widehat{kQ'} $ respectively, and
			\item[(3)] $ \varphi(W) $ equals $ W' $ in $ H\!H_{0}(\widehat{kQ'}) $.
		\end{itemize}
	\end{Def}
	
	
	The following lemma provides a normal form for irredundant potentials, up to right equivalence.
	\begin{Lem}\cite[Lemma 3.14]{presslandMutationFrozenJacobian2020}\label{Lem:Normal form of potential}
		Let $ (Q,F,W) $ be an ice quiver with potential such that W is irredundant. Then $ W $ admits a representative of the form
		\begin{equation}\label{Equ:potential decomp}
			\xymatrix{
				\widetilde{W}=\displaystyle\sum_{i=1}^{M}\alpha_{i}\beta_{i}+\sum_{i=M+1}^{N}(\alpha_{i}\beta_{i}+\alpha_{i}p_{i})+W_{1}
			}	
		\end{equation} 
		for some arrows $ \alpha_{i} $ and  $ \beta_{i} $ and elements $ p_{i}\in\mathfrak{m}^{2} $, where
		\begin{itemize}
			\item[(i)] $ \alpha_{i} $ is unfrozen for all $ 1\leqslant i\leqslant N $, and $ \beta_{i} $ is frozen if and only if $ i>M $. Then for $ 1\leqslant i\leqslant M $, the $ \alpha_{i}\beta_{i} $ are unfrozen 2-cycles and for $ M+1\leqslant i\leqslant N $, they are half frozen 2-cycles,
			\item[(ii)] the arrows $ \alpha_{i} $ and $ \beta_{i} $ with $ 1\leqslant i\leqslant M $ each appear exactly once in the expression~(\ref{Equ:potential decomp}),
			\item[(iii)] the arrows $ \beta_{i} $ for $ 1\leqslant i\leqslant N $, do not appear in any of the $ p_{j} $, and
			\item[(iv)] the arrow $ \alpha_{i} $ and $ \beta_{i} $, for $ 1\leqslant i\leqslant N $, do not appear in the term $ W_{1} $, and this term does not contain any 2-cycles.
		\end{itemize}
	\end{Lem}
	
	The following result allows us to replace any ice quiver with potential by a reduced one, without affecting the isomorphism class of the Jacobian algebra.
	
	\begin{Thm}\cite[Theorem 3.6]{presslandMutationFrozenJacobian2020}\label{Thm:reduction}
		Let $ (Q,F,W) $ be an ice quiver with potential. Then there exists a reduced ice quiver with potential $ (Q_{red} , F_{red},W_{red}) $ such that $ J(Q,F,W)\cong J(Q_{red} , F_{red},W_{red}) $.
	\end{Thm}

	\begin{Prop}\cite[Proposition 3.15]{presslandMutationFrozenJacobian2020}\label{Prop:rduction up to right equ}
		Let $ (Q,F,W) $ be an irredundant ice quiver with potential. Then the ice quiver with potential $ (Q_{red} , F_{red},W_{red}) $ from Theorem~\ref{Thm:reduction} is uniquely determined up to right equivalence by the right equivalence class of $ (Q,F,W) $.
	\end{Prop}

	\begin{Def}\rm
		Let $ (Q,F,W) $ is an irredundant ice quiver with potential. We call $ (Q_{red} , F_{red},W_{red}) $ from Theorem~\ref{Thm:reduction} the \emph{reduction} of $ (Q,F,W) $.
	\end{Def}

	\subsection{Algebraic mutations}
	Let $ (Q,F,W) $ be an ice quiver with an irredundant potential. Let $ v $ be an unfrozen vertex of $ Q $ such that no loops or 2-cycles of $ Q $ are incident with $ v $.
	
	\begin{Def}\rm\cite[Definition 4.1]{presslandMutationFrozenJacobian2020}\label{Def:Algebraic mutation}
		The ice quiver with potential $ \tilde{\mu}_{v}(Q,F,W) $, called the \emph{pre-mutation} of $ (Q,F,W) $ at $ v $, is the output of the following procedure.
		\begin{itemize}
			\item[(1)] For each pair of arrows $ \alpha:u\to v $ and $ \beta:v\to w $, add an unfrozen ‘composite' arrow $ [\beta\alpha]:u\to w $ to $ Q $.
			\item[(2)] Reverse each arrow incident with $ v $.
			
			\item[(3)] Pick a representative $ \widetilde{W} $ of $ W $ in $ kQ $ such that no term of $ W $ begins at $ v $ (which is possible since there are no loops at $ v $). For each pair of arrows $ \alpha,\beta $ as in $ (1) $, replace each occurrence of $ \beta\alpha $ in $ \widetilde{W} $ by $ [\beta\alpha] $, and add the term $ [\beta\alpha]\alpha^{*}\beta^{*} $.
			%
		\end{itemize}
		Let us write $ (Q',F',W') $ for $ \tilde{\mu}_{v}(Q,F,W) $. It is clear that $ F' $ equals $ F $ and the new potential $ W' $ is also irredundant, since the arrows $ [\beta\alpha] $ are unfrozen, but it need not be reduced even if $ (Q,F,W) $ is. We define $ \mu_{v}(Q,F,W) $ as replace the resulting ice quiver with potential $ \tilde{\mu}_{v}(Q,F,W) $ by its reduction, as in Theorem~\ref{Thm:reduction}, this being unique up to right equivalence by Proposition~\ref{Prop:rduction up to right equ}. We call $ \mu_{v} $ the \emph{mutation} at the vertex $ v $. 
	\end{Def}
	
	\begin{Thm}
		\begin{itemize}
			\item[a)] \cite[Proposition 4.2]{presslandMutationFrozenJacobian2020} The right equivalence class of $ \mu_{v}(Q,F,W) $ is determined by the right equivalence class of $ (Q,F,W) $.
			\item[b)]\cite[Theorem 4.3]{presslandMutationFrozenJacobian2020} Let $ (Q,F,W) $ be a reduced ice quiver with potential and $ v\in Q_{0}\setminus F_{0} $ an unfrozen vertex. Then $ \mu^{2}_{v}(Q, F, W) $ is right equivalent to $ (Q,F,W) $.
		\end{itemize}
	\end{Thm}

	\begin{Ex}
		Consider the following ice quiver with potential $ (Q,F,W) $
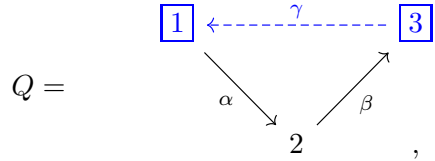
		\[
		\begin{tikzcd}[row sep=1ex]
			&\color{blue}\boxed{1}\arrow[ddr,swap,"\alpha"]&&\color{blue}\boxed{3}\arrow[ll,dashed,swap,blue,"\gamma"]\\
			Q=\\
			&&2\arrow[uur,swap,"\beta"]&,
		\end{tikzcd}
		\]
		where $ F $ is the full subquiver on $ \{1,3\}\subseteq Q_{0} $ and the potential $ W=\gamma\beta\alpha $. Then the pre-mutation of $ (Q,F,W) $ at vertex $ 2 $ is given by the following ice quiver with potential
		
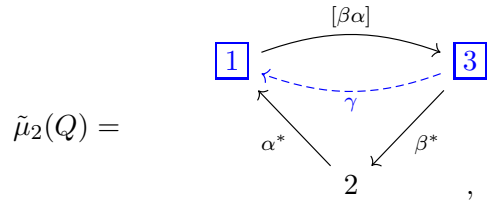
\[
		\begin{tikzcd}[row sep=1ex]
			&\color{blue}\boxed{1}\arrow[rr,"{[\beta\alpha]}",bend left=20]&&\color{blue}\boxed{3}\arrow[ll,dashed,blue,"\gamma",bend left=20]\arrow[ddl,"\beta^{*}"]\\
			\tilde{\mu}_{2}(Q)=\\
			&&2\arrow[uul,"\alpha^{*}"]&,
		\end{tikzcd}
		\]
		where $ F $ is the subquiver with vertex set $ F_{0}=\{1,2\} $ and arrow set $ F_{1}=\{\gamma\} $ and the new potential is $ \tilde{\mu}_{2}(W)=\alpha^{*}\beta^{*}[\beta\alpha]+\gamma[\beta\alpha] $. This ice quiver with potential is not reduced. The mutation $ \mu_{2}(Q,F,W) $ is given by its reduction, which is the ice quiver
		\[
		\begin{tikzcd}[row sep=1ex]
			&\color{blue}\boxed{1}\arrow[rr,dashed,blue,"{[\beta\alpha]}"]&&\color{blue}\boxed{3}\arrow[ddl,"\beta^{*}"]\\
			(\mu_{2}(Q),\mu_{2}(F))=\\
			&&2\arrow[uul,"\alpha^{*}"]&
		\end{tikzcd}
		\]
		with potential $ \mu_{2}W=\beta^{*}[\beta\alpha]\alpha^{*} $.
	\end{Ex}
	
	\begin{Thm}\cite[Proposition 4.6]{presslandMutationFrozenJacobian2020}
		Let $ (Q,F,W) $ be an ice quiver with an irredundant potential and $ v $ an unfrozen vertex. If $ (\mu_{v}Q,\mu_{v}F) $ has no $ 2 $-cycles containing unfrozen arrows, then the underlying ice quiver of $ \mu_{v}(Q,F,W) $ agrees with $ \mu^{P}_{v}(Q,F) $ defined in Definition~\ref{Def:Combinatorial mutations}.
	\end{Thm}
	
	\subsection{The complete relative Ginzburg algebra and the Ginzburg functor}
	\begin{Def}\rm\label{Def:Relative Ginzburg algebra}
		Let $ (Q,F,W) $ be a finite ice quiver with potential. Let $ \widetilde{Q} $ be the graded quiver with the same vertices as $ Q $ and whose arrows are
		\begin{itemize}
			\item the arrows of $ Q $,
			\item an arrow $ a^{\vee}:j\to i $ of degree $ -1 $ for each unfrozen arrow $ a $,
			\item a loop $ t_{i}:i\to i $ of degree $ -2 $ for each unfrozen vertex $ i $.
		\end{itemize}
		
		Define the \emph{complete relative Ginzburg dg algebra} $ \bm\Gamma_{rel}(Q,F,W) $ as the dg algebra whose underlying graded space is the completed graded path algebra $ \widehat{k\widetilde{Q}} $. Its differential is the unique $ k $-linear continuous endomorphism of degree 1 which satisfies the Leibniz rule
		\begin{align*}
			\xymatrix{
				d(u\circ v)=d(u)\circ v+(-1)^{p}u\circ d(v)
			}
		\end{align*}
		for all homogeneous $ u $ of degree $ p $ and all $ v $ and takes the following values on the arrows of $ \widetilde{Q} $:
		\begin{itemize}
			\item $ d(a)=0 $ for each arrow $ a $ of $ Q $,
			\item $d(a^{\vee})=\partial_{a}W$ for each unfrozen arrow $ a $,
			\item $ d(t_{i})=e_{i}(\sum_{a\in Q_{1}}[a,a^{\vee}])e_{i} $ for each unfrozen vertex $ i $, where $ e_{i} $ is the lazy path corresponding to the vertex $ i $.
		\end{itemize}
	\end{Def}

	\begin{Def}\rm\label{Def: derived preprojective algebra}
		Let $ F $ be any finite quiver. Let $  \widetilde{F} $ be the graded quiver with the same vertices as $ F $ and whose arrows are
		\begin{itemize}
			\item the arrows of $ F $,
			\item an arrow $ \tilde{a}:j\to i $ of degree $ 0 $ for each arrow $ a $ of $ F $,
			\item a loop $ r_{i}:i\to i $ of degree $ -1 $ for each vertex $ i $ of $ F $.
		\end{itemize}
		Define \emph{complete derived preprojective algebra} $ \bm\Pi_{2}(F) $ as the dg algebra whose underlying graded space is the completed graded path algebra $ \widehat{k\widetilde{F}} $. Its differential is the unique $ k $-linear continuous endomorphism of degree 1 which satisfies the Leibniz rule
		\begin{align*}
			\xymatrix{
				d(u\circ v)=d(u)\circ v+(-1)^{p}u\circ d(v)
			}
		\end{align*}
		for all homogeneous $ u $ of degree $ p $ and all $ v $, and takes the following values on the arrows of $ \widetilde{F} $:
		\begin{itemize}
			\item $ d(a)=0 $ for each arrow $ a $ of $ F $,
			\item $d(\tilde{a})=0 $ for each arrow $ a $ in $ F $,
			\item $ d(r_{i})=e_{i}(\sum_{a\in F_{1}}[a,\tilde{a}])e_{i} $ for each vertex $ i $ of $ F $, where $ e_{i} $ is the lazy path corresponding to the vertex $ i $.
		\end{itemize}
	\end{Def}
	
	\bigskip
	
	Let $ (Q,F,W) $ be a finite ice quiver with potential. Since $ W $ can be viewed as an element in $ HC_{0}(\widehat{kQ}) $, $ c=B(W) $ is an element in $ HH_{1}(\widehat{kQ}) $, where $ B $ is the Connes connecting map (see~\cite[Section 6.1]{kellerDeformedCalabiYau2011}). Let $ G:\widehat{kF}\hookrightarrow\widehat{kQ} $  be the canonical dg inclusion. Then $ \xi=(0,c) $ is an element of $ H\!H_{0}(G) $. 
	
	Via the deformed relative 3-Calabi--Yau completion of $ G $ with respect to the class $ \xi $, by Theorem~\ref{Thm:Relative CY completion to left steucture}, we get a dg functor
	$$ \bm \tilde{G}:\bm\Pi_{2}(F)\ra\bm\Pi_{3}(\widehat{kQ},\widehat{kF},\xi) ,$$ 
	which has a canonical left $ 3 $-Calabi--Yau structure. 

\begin{Prop}\cite[Proposition 3.18]{wuyilinHiggs}\label{Prop: reduce CY}
	We have the following commutative diagram (up to homotopy) in the category of pseudocompact dg algebras
	\[
	\begin{tikzcd}
		\bm\Pi_{2}(F)\arrow[r,"\tilde{G}"]\arrow[dr,"\bm G_{rel}",swap]&\bm\Pi_{3}(\widehat{kQ},\widehat{kF},\xi)\\
		&\bm{\Gamma}_{rel}(Q,F,W)\arrow[u,hook,"i"]
	\end{tikzcd}
	\]
	Where $ i $ is a quasi-equivalent dg inclusion and $ \bm G_{rel} $ is given explicitly as follows:
	\begin{itemize}
		\item $ \bm G_{rel}(i)=i $ for each frozen vertex $ i\in F_{0} $,
		\item $ \bm G_{rel}(a)=a $ for each arrow $ a\in F_{1} $,
		\item $ \bm G_{rel}(\tilde{a})=-\partial_{a}W $ for each arrow $ a\in F_{1} $,
		\item $ \bm G_{rel}(r_{i})=e_{i}(\sum_{a\in Q_{1}\setminus F_{1}}[a,a^{\vee}])e_{i} $ for each frozen vertex $ i\in F_{0} $.
	\end{itemize}
Thus, the functor $ \bm G_{rel}: \bm\Pi_{2}(F)\ra \bm{\Gamma}_{rel}(Q,F,W) $ has a canonical left $ 3 $-Calabi--Yau structure.
\end{Prop}
\begin{Rem}
	In the pseudocompact setting, the situation is different from \cite[Proposition 3.18]{wuyilinHiggs}. But the proof of Proposition~\ref{Prop: reduce CY} can be adapted from the proof of \cite[Proposition 3.18]{wuyilinHiggs}.
\end{Rem}
We call $ \bm G_{rel}:\bm\Pi_{2}(F)\ra\bm{\Gamma}_{rel}(Q,F,W) $ the \emph{Ginzburg functor} associated with $ (Q,F,W) $. The following lemma is an easy consequence of the definition.
	\begin{Lem}
		Let $ (Q,F,W) $ be an ice quiver with potential. Then the Jacobian algebra $ J(Q,F,W) $ is the $ 0 $-th homology of the complete Ginzburg dg algebra $ \bm{\Gamma}_{rel}(Q,F,W) $, i.e.
		$$ J(Q,F,W)=H^{0}(\bm\Gamma_{rel}(Q,F,W)) .$$ Moreover, the complete preprojective algebra $ \widetilde{kF} $ is the $ 0 $-th homology of the complete derived preprojective algebra $ \bm\Pi_{2}(F) $, i.e.
		$$ \widetilde{kF}=H^{0}(\bm\Pi_{2}(F)) .$$
	\end{Lem}

	Let $ (Q,F,W) $ be an ice quivers with potential. Let $ \bm\Gamma_{rel}=\bm\Gamma_{rel}(Q,F,W) $ be the associated complete relative Ginzburg dg algebra. Let $ (Q_{red},F_{red},W_{red}) $ be the reduction of $ (Q,F,W) $ as in Theorem~\ref{Thm:reduction}.
	
	\begin{Lem}\label{Lem:irredundant to dg}
		There is an irredundant potential $ W' $ such that $ \bm\Gamma_{rel}(Q,F,W)\cong\bm\Gamma_{rel}(Q,F,W') $.
	\end{Lem}
	\begin{proof}
		We collect all terms containing only frozen arrows. Then there is a unique decomposition $ W=W'+W_{F} $ in which $ W' $ is irredundant and $ W_{F} $ is a potential on $ F $. Since $ \partial_{a}W_{F} $ is 0 for any arrow $ a\in Q_{1}\setminus F_{1} $, it is clear that $ \bm\Gamma_{rel}(Q,F,W) $ is isomorphic to $ \bm\Gamma_{rel}(Q,F,W') $.
	\end{proof}
	
	\begin{Lem}\label{Lem:right-equivalent to dg}
		Let $ (Q',F',W') $ be another ice quivers with potential. We denote by $ \bm\Gamma'_{rel}=\bm\Gamma_{rel}(Q',F',W') $ the associated complete relative Ginzburg dg algebra. Assume that $ (Q,F,W) $ and $ (Q',F',W') $ are right-equivalent. Then $ \bm\Gamma_{rel} $ and $ \bm\Gamma'_{rel} $ are isomorphic to each other.
	\end{Lem}
	\begin{proof}
		The proof is the same as that of~\cite[Lemma 2.9]{kellerDerivedEquivalencesMutations2011}.
	\end{proof}
	%
	%
	
	\begin{Lem}
		There exists a quasi-isomorphism between the complete relative Ginzburg algebras
		$$ \bm\Gamma_{rel}(Q,F,W)\ra\bm\Gamma_{rel}(Q_{red},F_{red},W_{red}) .$$
	\end{Lem}	
	\begin{proof}
		By Lemma~\ref{Lem:irredundant to dg}, we can assume that $ W $ is irredundant. By Lemmas~\ref{Lem:Normal form of potential} and~\ref{Lem:right-equivalent to dg}, we can assume that $ W $ has an expression of the form~(\ref{Equ:potential decomp}) satisfying conditions (i)–(iv), i.e.$ \, $the potential $ W $ can be written as
		$$ \xymatrix{
			W=\displaystyle\sum_{i=1}^{M}\alpha_{i}\beta_{i}+\sum_{i=M+1}^{N}(\alpha_{i}\beta_{i}+\alpha_{i}p_{i})+W_{1}
		} $$ as in (\ref{Equ:potential decomp}).
		Take $ Q' $ to be the subquiver of $ Q $ consisting of all vertices and the arrows $ \alpha_{i}, \beta_{i} $ for $ i\leqslant M $ and $ W'=\sum_{i=1}^{M}\alpha_{i}\beta_{i} $.  Then the quiver with potential $ (Q',W') $ is trivial.
		
		Let $ Q'' $ be the subquiver of $ Q $ consisting of all vertices and those arrows that are not in $ Q_{1}' $. Let $ W''=W-W'=\sum_{i=M+1}^{N}(\alpha_{i}\beta_{i}+\alpha_{i}p_{i})+W_{1} $. We see that $ W'' $ does not involve any arrows of $ Q' $. Thus, this defines a potential on $ Q'' $. As in~\cite[Lemma 2.10]{kellerDerivedEquivalencesMutations2011}, we see that the canonical projection $ \bm\Gamma_{rel}(Q,F,W)\ra\bm\Gamma_{rel}(Q'',F,W'') $ is a quasi-isomorphism. Simplifying the expression for $ W'' $ and relabeling arrows for simplicity, we have
		$$ W''=\sum_{i=1}^{K}(\alpha_{i}\beta_{i}+\alpha_{i}p_{i})+W_{1}=\sum_{i=1}^{K}\alpha_{i}\beta_{i}+W_{red} ,$$ where each $ \alpha_{i} $ is unfrozen, each $ \beta_{i} $ is frozen and does not appear in any term of $ W_{red} $. This is ensured by the condition (iii) in Lemma~\ref{Lem:Normal form of potential}. 
		
		By the proof of~\cite[Theorem 3.6]{presslandMutationFrozenJacobian2020}, the ice quiver $ (Q_{red},F_{red}) $ is obtained from $ (Q'',F) $ by deleting $ \beta_{i} $ and freezing $ \alpha_{i} $ for each $ 1\leqslant i\leqslant K $. Then the ice quiver with potential $ (Q_{red},F_{red},W_{red}) $ is the reduction of $ (Q,F,W) $.
		
		Let $ M_{red} $ be the pseudocompact dg $ kQ_{0}'' $-bimodule generated by 
		$$ S_{red}=\{\gamma,\delta^{\vee},t_{i}\,|\,\gamma\in (Q_{red})_{1},\,\delta\in (Q_{red})_{1}\setminus (F_{red})_{1},i\in (Q_{red})_{0}\setminus (F_{red})_{0}\} .$$ Let $ M'' $ be the pseudocompact dg $ kQ''_{0} $-bimodule generated by
		$$ S''=S_{red}\cup\{\beta_{i},\alpha_{i}^{\vee}\,|\,1\leqslant i\leqslant K\} .$$
		By the construction of relative Ginzburg dg algebras, the underlying graded categories have the forms
		$$ \bm\Gamma_{rel}(Q_{red},F_{red},W_{red})=T_{kQ''_{0}}(M_{red}) $$ and 
		$$ \bm\Gamma_{rel}(Q'',F,W'')=T_{kQ''_{0}}(M'') .$$
		It is easy to see that we have an inclusion $ i:M_{red}\hookrightarrow M'' $ of dg $ kQ''_{0} $-bimodules. Then $ i $ induces a fully faithful morphism $ i:\bm\Gamma_{rel}(Q_{red},F_{red},W_{red})\hookrightarrow\bm\Gamma_{rel}(Q'',F,W'') $ in $ \mathrm{PcAlgc}(kQ''_{0}) $. We define another dg $ kQ''_{0} $-bimodule morphism $ \varphi:M''\ra M_{red} $ as follows:
		\begin{itemize}
			\item $ \varphi(\gamma)=\gamma $ for any arrow $ \gamma $ in $ Q''_{1} $ such that $ \gamma\neq\beta_{i} $ with $ 1\leqslant i\leqslant K $,
			\item $ \varphi(\gamma^{\vee})=\gamma^{\vee} $ for any arrow $ \gamma $ in $ Q''_{1}\setminus F_{1} $ such that $ \gamma\neq\alpha_{i} $ with $ 1\leqslant i\leqslant K $,
			\item $ \varphi(\beta_{i})=-\partial_{\alpha_{i}}W_{red} $ for each $ 1\leqslant i\leqslant K $,
			\item $ \varphi(\alpha_{i}^{\vee})=0 $ for each $ 1\leqslant i\leqslant K $,
			\item $ \varphi(t_{i})=t_{i} $ for each unfrozen vertex $ i $.
		\end{itemize}
		Then $ \varphi $ induces a morphism $ f:\bm\Gamma_{rel}(Q'',F,W'')\ra\bm\Gamma_{rel}(Q_{red},F_{red},W_{red}) $ in $ \mathrm{PcAlgc}(kQ''_{0}) $. To see that $ f $ is a quasi-isomorphism, it is enough to show that $ f $ is a homotopy inverse of $ i $. 
		
		It is clear that we have $ f\circ i=\id $. We define a continuous morphism $ h:\bm\Gamma_{rel}(Q'',F,W'')\ra\bm\Gamma_{rel}(Q'',F,W'') $ of graded $ k $-modules which is homogeneous of degree -1, satisfies $$ h(xy)=h(x)f(y)+(-1)^{-n}xh(y) $$ for all $ x\in\bm\Gamma^{(n)}_{rel}(Q'',F,W'') $, $ y\in\bm\Gamma_{rel}(Q'',F,W'') $ and
		
		\begin{itemize}
			\item $ h(\beta_{i})=\alpha_{i}^{\vee} $ for each $ 1\leqslant i\leqslant K $,
			\item $ h(\delta)=0 $ for all other arrows $ \delta $.
		\end{itemize}
		Then we have $ \id-i\circ f =d(h) $. Thus, the morphism $ f $ is a homotopy inverse of $ i $.
	\end{proof}

	\subsection{Cofibrant resolutions of simples over a tensor algebra}\label{subsection:Cofibrant resolutions of simples over a tensor algebra}
	Let $ Q $ be a finite graded quiver and $ \widehat{kQ} $ the complete path algebra. Let $ \mathfrak{m} $ be the two-sided ideal of $ \widehat{kQ} $ generated by arrows of $ Q $. Let $ A=(\widehat{kQ},d) $ be a pseudocompact dg algebra whose differential takes each arrow of $ Q $ to an element of $ \mathfrak{m} $.
	
	For a vertex $ i $ of $ Q $, let $ P_{i}=e_{i}A $, and let $ S_{i} $ be the simple module corresponding to $ i $. Then we have a short exact sequence in $ \cc(A) $
	$$ 0\ra \ker(\pi)\ra P_{i}\xrightarrow{\pi} S_{i}\ra 0 ,$$ where $ \pi $ is the canonical projection from $ P_{i} $ to $ S_{i} $. More explicitly, the graded $ A $-module $ \ker(\pi) $ decomposes as
	\begin{equation*}
		\ker(\pi)=\bigoplus_{\rho\in Q_{1}:t(\rho)=i}\rho P_{s(\rho)}	
	\end{equation*} 
	with the induced differential. Thus, the simple module $ S_{i} $ is quasi-isomorphic to
	$$ P=\cone(\ker(\pi)\ra P_{i}) ,$$ whose underlying graded space is
	$$ \bigoplus_{\rho\in Q_{1}:t(\rho)=i}\Si\rho P_{s(\rho)}\oplus P_{i} .$$ By~\cite[Section 2.14]{kellerDerivedEquivalencesMutations2011}, the dg module $ P $ is a cofibrant dg $ A $-module and hence it is a cofibrant resolution of $ S_{i} $. In particular, the simple dg module $ S_{i} $ belongs to the perfect derived category $ \per(A) $.
	

	\section{Main results}\label{Section5}
	
	\subsection{Compatibility with Morita functors and localizations}
	Let $ R $ be a finite dimensional separable $ k $-algebra. Let $ J:A\to A' $ a morphism in $ \mathrm{PcAlgc}(R) $. We say that $ J $ is a \emph{localization functor} if the derived functor $ J^{*} $ induces an equivalence
	\begin{align*}
		J^{*}:\mathcal{D}(A)/\mathcal{N}\iso\mathcal{D}(A')
	\end{align*}
	for some localizing subcategory $ \mathcal{N} $ of $ \mathcal{D}(A) $. Equivalently, the restriction functor $ J_{*}:\mathcal{D}(A')\to\mathcal{D}(A) $ is fully faithful.
	
	We say that $ J:A\to A' $ is a \emph{Morita functor} if restriction functor $ J_{*} $ is an equivalence from $ \mathcal{D}(A') $ to $ \mathcal{D}(A) $. Equivalently, the derived functor $ J^{*} $ is an equivalence.

	\begin{Thm}\label{Thm:Derived equivalence}
		Let $ S $ be another finite dimensional separable $ k $-algebra and $ B $ an object in $ \mathrm{PcAlgc}(S) $. Let $ I:B\ra B' $ be a localization functor in $ \mathrm{PcAlgc}(S) $ and $ J:A\ra A' $ a localization functor in $ \PcAlgc $. Suppose that we have morphisms $ f:B\ra A $ and $ f':B'\ra A' $ such that the following square commutes in the category of dg $ k $-algebras
		\begin{align}\label{Diagram:commutative}
			\xymatrix{
				B\ar"1,3"^{f}\ar[d]_{I}&&A\ar[d]^{J}\\
				B'\ar"2,3"^{f'}&&A'.
			}
		\end{align}
		Assume that $ B $, $ A $, $ B' $ and $ A' $ are smooth and connective. Let $ [\xi]=[(s\xi_{B},\xi_{A})] $ be an element of $ H\!H_{n-2}(f) $ and $ [\xi']=[(s\xi_{B'},\xi_{A'})] $ the element of $ H\!H_{n-2}(f') $ obtained as the image of $ [\xi] $ under the map induced by $ I $ and $ J $. Then we have the following commutative diagram in the category of dg $ k $-algebras
		\begin{align*}
			\xymatrix{
				B\ar@{^{(}->}[r]\ar[d]_{I}&\bm{\Pi}_{n-1}(B,\xi_{B})\ar[r]^{\widetilde{f}}\ar[d]_{I'}&\bm{\Pi}_{n}(A,B,\xi)\ar[d]_{J'}\\
				B'\ar@{^{(}->}[r]&\bm{\Pi}_{n-1}(B',\xi_{B'})\ar[r]^{\widetilde{f'}}&\bm{\Pi}_{n}(A',B',\xi')
			}
		\end{align*}
		where $ I' $, $ J' $ are also localization functors. Moreover, $ I' $ (respectively, $ J' $) is a Morita functors if $ I $ (respectively, $ J $) is. 
	\end{Thm}
	\begin{proof} 
		By~\cite[Thm 5.8]{kellerDeformedCalabiYau2011}, there exists a canonical localization functor $ I' $ such that the leftmost square above commutes and $ I' $ is a Morita functor if $ I $ is. We use the same method as in~\cite[Thm 5.8]{kellerDeformedCalabiYau2011} to show the existence of $ J' $ and the commutativity of the rightmost square.
		
		Let $ P_{B} $, $ P_{A} $, $ P_{A'} $  and $ P_{B'} $ be the canonical bar resolutions of $ B $, $ A $, $ B' $ and $ A' $ as pseudocompact bimodules over themselves respectively (see~\cite[Lemma B.2]{vandenberghCalabiYauAlgebrasSuperpotentials2015}). We denote by
		$ j_{f}:f^{*}(P_{B})\to P_{A} $, $ j_{f'}:f'^{*}(P_{B'})\to P_{A'} $, $ j_{I}:I^{*}(P_{B})\to P_{B'} $ and $ j_{J}:J^{*}(P_{A})\to P_{A'} $ the canonical morphisms induced by $ f $, $ f' $, $ I $ and $ J $ respectively. 
		
		We denote by $ k_{I}:P_{B}\to I_{*}(P_{B'}) $ the canonical morphism induced by the adjunction $ (I^{*},I_{*}) $. Then we have a commutative diagram in $ \cc_{pc}(A^{e}) $
		\begin{align}\label{Morita comm}
			\xymatrix{
				P_{A}^{\vee}\ar[r]^{j_{f}^{\vee}}\ar[d]^{j_{J}^{\vee}}&f^{*}(P_{B}^{\vee})\ar[r]^{l_{f}}\ar[d]^{f^{*}(k_{I})^{\vee}}&\Xi_{f}\ar[d]^{r}\\
				J_{*}(P_{A'}^{\vee})\ar[r]^-{J_{*}(j_{f'})^{\vee}}&J_{*}f'^{*}(P_{B'}^{\vee})=f^{*}I_{*}(P_{B'}^{\vee})\ar[r]^-{J_{*}(l_{f'})}&J_{*}(\Xi_{f'}),
			}
		\end{align}
		where $ \Xi_{f} $ is the mapping cone of $ j_{f}^{\vee} $ and $ \Xi_{f'} $ is the mapping cone of $ j_{f'}^{\vee} $.
		
		By the commutative diagram~(\ref{Diagram:commutative}), we have the following commutative diagram of Hochschild complexes
		\begin{align*}
			\xymatrix{
				H\!H(B)\ar[r]^{f}\ar[d]^{I}&H\!H(A)\ar[d]^{J}\ar[r]&H\!H(f)\ar[d]\\
				H\!H(B')\ar[r]^{f'}&H\!H(A')\ar[r]&H\!H(f').
			}
		\end{align*}
		Thus, we have $ \xi_{B'}=I(\xi_{B}) $ and $ \xi_{A'}=J(\xi_{A}) $.
		
		The Hochschild homology classes $ \xi_{B} $ and $ \xi_{A} $ yield the following two commutative squares in $ \cc_{pc}(B^{e}) $ and $ \cc_{pc}(A^{e}) $ respectively
		
		\begin{align}\label{Morita xiA}
			\xymatrix{
				P_{B}^{\vee}\ar[r]^{\xi_{B}}\ar[d]&B\ar[d]\\
				I_{*}(P_{B'}^{\vee})\ar[r]^{I(\xi_{B})}&I_{*}(B'),
			}
		\end{align}
		
		\begin{align}\label{Morita xiB}
			\xymatrix{
				P_{A}^{\vee}\ar[r]^{\xi_{A}}\ar[d]&A\ar[d]\\
				J_{*}(P_{A'}^{\vee})\ar[r]^{J(\xi_{A})}&J_{*}(A').
			}
		\end{align}
		Then we get the following commutative diagram in $ \cc_{pc}(A^{e}) $
		\begin{align}\label{Morita Xi}
			\xymatrix{
				\Xi_{f}\ar[r]^{\xi}\ar[d]&\mathcal{A}\ar[d]\\
				J_{*}(\Xi_{f'})\ar[r]^{\xi'}&J_{*}(A').
			}
		\end{align}
		
		Combining the commutative diagrams (\ref{Morita comm}), (\ref{Morita xiA}), (\ref{Morita Xi}) and the proofs in \cite[Theorem 4.6, Theorem 5.8]{kellerDeformedCalabiYau2011}, we obtain the following commutative diagram in the category of dg $ k $-algebras.
		\begin{align*}
			\xymatrix{
				B\ar@{^{(}->}[r]\ar[d]_{I}&\bm{\Pi}_{n-1}(B,\xi_{B})\ar[r]^{\widetilde{f}}\ar[d]_{I'}&\bm{\Pi}_{n}(A,B,\xi)\ar[d]_{J'}\\
				B'\ar@{^{(}->}[r]&\bm{\Pi}_{n-1}(B',\xi_{B'})\ar[r]^{\widetilde{f'}}&\bm{\Pi}_{n}(A',B',\xi').
			}
		\end{align*}
		It remains to be shown that the restriction functor $ J'_{*}:\mathcal{D}(\bm{\Pi}_{n}(A',B',\xi'))\ra\mathcal{D}(\bm{\Pi}_{n}(A,B,\xi)) $ induced by $ J' $ is fully faithful.
		
		Let $ M $ be a right $ \bm{\Pi}_{n}(A',B',\xi') $-module and suppose that $ M $ is cofibrant.
		It is given by its underlying right $ A' $-module and a morphism of graded modules homogeneous of degree 0
		\begin{align*}
			\xymatrix{
				\lambda:M\otimes\Xi_{f'}[n-2]\ar[r]&M
			}
		\end{align*}
		such that $ (d\lambda)(m\otimes x)=m\xi'(x)$ for all $ m\in M $ and $ x\in\Xi_{f'}[n-2] $.
		Since $ \bm{\Pi}_{n}(A',B',\xi') $ is also cofibrant as right $ A' $-module, the underlying $ A' $-module of $ M $ is also cofibrant. Then we have an exact sequence of cofibrant $ \bm{\Pi}_{n}(A',B',\xi') $-modules
		\begin{align*}
			\xymatrix{
				0\ar[r]&M\otimes_{A'}\Xi_{f'}[n-2]\otimes_{A'}T_{A'}(\Xi_{f'}[n-2])\ar[r]^-{\alpha}&M\otimes_{A'}T_{A'}(\Xi_{f'}[n-2])\ar[r]&M\ar[r]&0
			}
		\end{align*}
		where $ \alpha(m\otimes x\otimes u)=mx\otimes u-m\otimes xu $. 
		
		This shows that the cone over the following morphism
		\begin{align*}
			\xymatrix{
				M\otimes_{A'}\Xi_{f'}[n-2]\otimes_{A'}T_{A'}(\Xi_{f'}[n-2])\ar[r]^-{\alpha}&M\otimes_{A'}T_{A'}(\Xi_{f'}[n-2])
			}
		\end{align*}
		is quasi-equivalent to $ M $. Let $ N $ be another right $ \bm{\Pi}_{n}(A',B',\xi') $-module. Then $ \Hom_{\mathcal{D}(\bm{\Pi}_{n}(A',B',\xi'))}(M,N) $ can be computed as the cone of the following morphism of dg $ k $-modules
		\begin{align*}
			\xymatrix{
				\Hom_{A'}(M,N)\ar[r]&\Hom_{A'}(M\otimes_{A'}\Xi_{f'}[n-2],N)
			}
		\end{align*}	
		Similarly,  $ \Hom_{\mathcal{D}(\bm{\Pi}_{n}(A,B,\xi))}(J'_{*}(M),J'_{*}(M)) $ can be computed analogously.
		Thus, it suffices to check that for all $ N $, the dg functor $ J $ induces the following bijections
		\begin{align*}
			\xymatrix{
				\Hom_{\mathcal{D}(A')}(M,N)\ar[r]&\Hom_{\mathcal{D}(A)}(J_{*}(M),J_{*}(N))
			}
		\end{align*}
		and
		\begin{align*}
			\xymatrix{
				\Hom_{\mathcal{D}(A')}(M\otimes^{\mathbi{L}}_{A'}\Xi_{f'}[n-2],N)\ar[r]&\Hom_{\mathcal{D}(A)}(J_{*}(M)\otimes^{\mathbi{L}}_{A}\Xi_{f}[n-2],J_{*}(N)).
			}
		\end{align*}
		The first bijection follows from the full faithfulness of $ J_{*} $. 
		For the second one, it is enough to show that the following two morphisms are bijections
		\begin{align*}
			\xymatrix{
				\Hom_{\mathcal{D}(A')}(M\otimes^{\mathbi{L}}_{A'}A'^{\vee},N)\ar[r]&\Hom_{\mathcal{D}(A)}(J_{*}(M)\otimes^{\mathbi{L}}_{A}A^{\vee},J_{*}(N))
			}
		\end{align*}
		
		\begin{align*}
			\xymatrix{
				\Hom_{\mathcal{D}(A')}(M\otimes^{\mathbi{L}}_{A'}f'^{*}(B'^{\vee}),N)\ar[r]&\Hom_{\mathcal{D}(A)}(J_{*}(M)\otimes^{\mathbi{L}}_{A}f^{*}(B^{\vee}),J_{*}(N))
			}.
		\end{align*}
		Indeed, this is a consequence of the full faithfulness of $ J_{*} $ and of the following formulas (see~\cite[Proposition 3.10]{kellerDeformedCalabiYau2011}) 
		\begin{align*}
			J^{*}(J_{*}(M)\otimes^{\mathbi{L}}_{A}A^{\vee})\xrightarrow{\sim}M\otimes^{\mathbi{L}}_{A'}A'^{\vee}\,,\, B'^{\vee}=I^{*}(B^{\vee})
		\end{align*}
		and 
		\begin{align*}
			J^{*}(J_{*}M\otimes_{A}^{\mathbi{L}}f^{*}(B^{\vee}))\xrightarrow{\sim}M\otimes_{A'}^{\mathbi{L}}f'^{*}(
			B'^{\vee}).
		\end{align*}
		
		\bigskip	
		If $ J $ is a Morita functor, by part (e) of~\cite[Proposition 3.10]{kellerDeformedCalabiYau2011}, the morphisms $ j_{J}^{\vee} $ and $ k_{I}^{\vee} $ in diagram~\ref{Morita comm} are quasi-isomorphisms. Then we see that $ r:\Xi_{f}\to J_{*}(\Xi_{f'}) $ is a quasi-isomorphism in $ \cc_{pc}(A^{e}) $. It induces quasi-isomorphisms between the tensor powers
		\begin{align*}
			\xymatrix{
				\Xi_{f}^{\cten^{n}}\ar[r]& J_{*}(\Xi_{f'}^{\cten^{n}})
			}
		\end{align*}
		for all $ n\geqslant 1 $. Thus,  $ r $ and $J $ yield a morphism
		\begin{align*}
			\xymatrix{
				J':\bm\Pi_{n}(A,B,\xi)\ar[r]& \bm\Pi_{n}(A',B',\xi')
			}
		\end{align*}	
		which is a quasi-equivalence. Thus, the morphism $ J' $ is a Morita functor.

		%
	\end{proof}
	
	\subsection{Derived equivalences}\label{Subsection:Derived equivalences}
	Let $ (Q,F,W) $ be a finite ice quiver with potential. For a vertex $ i $ of $ Q $, let $ P_{i}=e_{i}\widehat{kQ} $.
	
	Let $ v $ be an unfrozen vertex of $ Q $ such that no loops or 2-cycles are incident with $ v $. We assume that $ v $ is the source of at least one arrow.
Let $ T=\displaystyle\bigoplus_{\substack{j\neq v}}P_{j}\oplus T_{v}$, where $ T_{v} $ is defined by the exact sequence
	\[
	\begin{tikzcd}
		0\arrow[r]& P_{v}\arrow[r]&\displaystyle\bigoplus_{\alpha\in Q_{1}:s(\alpha)=v}P_{j}\arrow[r]& T_{v}\arrow[r]&0,
	\end{tikzcd}
	\]
	where the sum is taken over all arrows $ \alpha:v\to j $ and the corresponding component of the map from $ P_{v} $ to the sum is the left multiplication by $ \alpha $.
	It is easy to see that $ T $ is a tilting object in $ \cd(\widehat{kQ}) $. Let $ A' $ be the pseudocompact algebra $ \End_{\widehat{kQ}}(T) $. 
	
	\begin{Thm}\cite{kellerDerivingDGCategories1994a}\label{Thm:object to derived equivalence}
		Let $ A $ be a dg algebra and $ T $ an object of $ \cd(A) $. Denote by $ A' $ the dg algebra $ \RHom_{A}(T, T) $. Denote by $ \langle T\rangle_{A} $ the thick subcategory of $ \cd(A) $ generated by $ T $. Then the functor $ \RHom_{A}(T,?):\cd(A)\ra\cd(A') $ induces a triangle equivalence
		$$ \RHom_{A}(T,?):\langle T\rangle_{A}\ra\per A' .$$
	\end{Thm}
	
     Since $ T $ is a tilting object in $ \cd(\widehat{kQ}) $, by the above Theorem~\ref{Thm:object to derived equivalence}, we have a triangle equivalence
	$$ ?\lten_{A'}T:\per A'\ra\langle T\rangle_{\widehat{kQ}}\simeq\per\,\widehat{kQ} .$$ 
	
	Thus, the pseudocompact algebras $ A' $ and $ \widehat{kQ} $ are Morita equivalent.
	
	Let $ \ca $ be the full dg subcategory of $ \cc_{pc}^{dg}(\widehat{kQ}) $ whose objects are $ T_{v} $ and the $ P_{i} $, $ i\neq v $. Then $ \ca $ is Morita equivalent to $ \widehat{kQ} $. Let $ \ca' $ be the full dg subcategory of $ \cc_{pc}^{dg}(A') $ whose objects are the $ P'_{i}=e_{i}A' $, $ i\in Q_{0} $. Then $ \ca' $ is equal to $ A' $. We define a dg functor
	$$ J:\ca'\ra\ca $$ as follows:
	\begin{itemize}
		\item For $ k\neq v $, we put $ J(P'_{k})=P_{k} $,
		\item For $ k=v $, we put $ J(P'_{v})=T_{v} $. 
	\end{itemize}
	Then $ J $ is a Morita functor. It induces an isomorphism $ HC_{0}(A')\simeq HC_{0}(\widehat{kQ}) $. Let $ W'\in HC_{0}(A') $ be the element corresponding to $ W\in HC_{0}(\widehat{kQ}) $. We set $ \xi=B(W)\in HH_{1}(\widehat{kQ}) $ and $ \xi'=B(W')\in HH_{1}(A') $ where $ B $ is the Connes connecting map.
	
	Let $ \cb $ be the full dg subcategory of $ \cc_{pc}^{dg}(\widehat{kF}) $ whose objects are $ P_{i}=e_{i}\widehat{kF} $, $ i\in F_{0} $. Since the vertex $ v $ is unfrozen, we have dg inclusions $ G:\cb\hookrightarrow\ca $ and $ G':\cb\hookrightarrow\ca' $. Moreover, the following diagram commutes
	\begin{align*}
		\xymatrix{
			&&\ca'\ar"3,3"^{J}\\
			\cb\,\ar@{^{(}->}"1,3"^{G'}\ar@{^{(}->}"3,3"_{G}&&\\
			&&\ca\,.
		}
	\end{align*}
	
	Applying the deformed relative 3-Calabi-Yau completion to the above diagram with respect to the potentials $ W' $ and $ W=J(W') $, we get the following commutative diagram of pseudocompact dg algebras
	\begin{align}\label{diag: J-morita}
		\xymatrix{
			&&\bm\Pi_{3}(\ca',\cb,\xi')\ar[dd]^{J}\\
			\bm{\Pi}_{2}(F)\ar[urr]^{\tilde{G'}}\ar[drr]_{\bm{G}_{rel}}\\
			&&\bm{\Gamma}_{rel}(Q,F,W).
		}
	\end{align}
	By Theorem~\ref{Thm:Derived equivalence}, the dg functor $ J $ is a Morita functor. 
	
	\bigskip

	For each arrow $ \beta:v\ra w $ in $ Q $, we denote by $ \beta^{*} $ the canonical map from $ P_{w} $ to $ T_{v} $. For each pair of arrows $ (\alpha,\beta) $ with $ \alpha:u\to v $ and $ \beta:v\to w $, we have an arrow $ [\beta\alpha]:P_{u}\ra P_{w} $ in the quiver of $ A' $. 
	
	Then for each arrow $ \alpha:u\ra v $, we have a minimal relation in $ A' $ which is given by 
	$$ \sum_{\beta\in Q_{1}:s(\beta)=v}\beta^{*}[\beta\alpha] .$$
	It is easy to see that all minimal relations in $ A' $ come from this way.

    Let $ Q_{A'} $ be the quiver of $ A' $. Let $ Q' $ be the quiver with same vertices as $ Q_{A'} $ and whose arrows are
    \begin{itemize}
    	\item The arrows of $ Q_{A'} $,
    	\item an arrow $ \alpha^{*}:T_{v}\ra P_{u} $ for each arrow $ \alpha:P_{u}\ra P_{v} $ in $ Q_{A'} $.
    \end{itemize}

	 Now we define a potential $ W_{1} $ on $ Q' $ by $ W_{1}=\sum_{\beta\in Q_{1}:s(\beta)=v}\alpha^{*}\beta^{*}[\beta\alpha] $ and a potential $ W_{2} $ by lifting $ W' $ along the surjection $ kQ'\to A' $ taking all arrows $ \alpha^{*} $ to zero. Here $ W'\in HC_{0}(A') $ is the element corresponding to $ W\in HC_{0}(\widehat{kQ}) $ under the isomorphism $ HC_{0}(A')\simeq HC_{0}(\widehat{kQ}) $. Hence $ W_{2} $ is is obtained from $ W $ by replacing each occurrence of a composition $ \beta\alpha $ in a cycle passing through $ v $ by $ [\beta\alpha] $ (see \cite[Section 7.6]{kellerDeformedCalabiYau2011}). 
	
	One can check that the ice quiver with potential $ (Q',F,W_{1}+W_{2}) $ is the pre-mutation (Definition~\ref{Def:Algebraic mutation}) of $ (Q,F,W) $ at the vertex $ v $, i.e. $ \tilde{\mu}_{v}(Q,F,W)=(Q',F,W_{1}+W_{2}) $ (also see~\cite[Subsection 7.6]{kellerDeformedCalabiYau2011}).

	\bigskip
	\begin{Prop}~\cite[Theorem 6.10]{kellerDeformedCalabiYau2011}\label{Prop: CY completion is Ginzburg}
		The pseudocompact dg algebra $ \bm\Pi_{3}(\ca',\cb,W') $ is quasi-isomorphic to the complete relative Ginzburg algebra $ \bm\Gamma_{rel}(Q',F,W_{1}+W_{2}) $.
	\end{Prop}
	\begin{proof}
	We first construct a pseudocompact dg path algebra $ C $ and a quasi-isomorphism $ \delta:C\ra \ca' $. The underling graded path quiver of $ C $ has the same vertices as $ Q_{A'} $ and whose arrows are
	\begin{itemize}
		\item The arrows of $ Q_{A'} $,
		\item an arrow $ \epsilon_{\alpha}:P_{u}\ra T_{v} $ of degree $ -1 $ for each arrow $ \alpha:u\ra v $ in $ Q_{A'} $.
	\end{itemize}
	The differential of $ C $ is defined by $ d_{C}(\epsilon_{\alpha})=\sum_{\beta\in Q_{1}:s(\beta)=v}\beta^{*}[\beta\alpha] $. Then the canonical projection map $ \delta:C\ra A' $ is a quasi-isomorphism. Moreover, we have the following commutative diagram of pseudocompact dg algebras
	\begin{align*}
		\xymatrix{
			&&C\ar"3,3"^{\delta}\\
			\cb\,\ar@{^{(}->}"1,3"^{H}\ar@{^{(}->}"3,3"_{G'}&&\\
			&&\ca'=A'\,.
		}
	\end{align*}
	The quasi-isomorphism $ \delta $ induces an isomorphism $ H\!C_{0}(C)\iso H\!C_{0}(\ca') $. Let $ W_{C}\in HC_{0}(C) $ be the element corresponding to $ W'\in HC_{0}(\ca') $. We set $ \xi_{C}=B(W_{C})\in HH_{1}(C) $.
	
	The deformed relative 3-Calabi--Yau completion of $ \cb\ra\ca' $ with respect to $ \xi'=B(W')\in HH_{1}(A') $ is quasi-isomorphic to the deformed relative 3-Calabi--Yau completion of $ \cb\ra C $ with respect to $ \xi_{C}=B(W_{C})\in HH_{1}(C) $, i.e.\ we have the following commutative diagram of pseudocompact dg algebras
	\begin{align}\label{diag: delta iso}
		\xymatrix{
			&&\bm\Pi_{3}(C,\cb,\xi_{C})\ar[dd]^{\tilde{\delta}}\\
			\bm{\Pi}_{2}(F)\ar[urr]^{\bm{G}'_{rel}}\ar[drr]_{\tilde{G'}}\\
			&&\bm\Pi_{3}(\ca',\cb,\xi'),
		}
	\end{align}
 where $ \tilde{\delta} $ is a quasi-isomorphism. An explicit computation shows that $ \bm\Pi_{3}(C,\cb,\xi_{C}) $ is equal to the complete relative Ginzburg algebra $ \bm\Gamma_{rel}(Q',F,W_{1}+W_{2}) $.

\end{proof}

Combining the results in commutative diagrams (\ref{diag: J-morita}) and (\ref{diag: delta iso}), we get the following commutative diagram of pseudocompact dg algebras
\begin{align}\label{diag: mutation to morita}
	\xymatrix{
		&&\bm\Gamma_{rel}(Q',F,W_{1}+W_{2})\ar[dd]^{J\circ\tilde{\delta}}\\
		\bm{\Pi}_{2}(F)\ar[urr]^{\bm{G}'_{rel}}\ar[drr]_{\bm{G}_{rel}}\\
		&&\bm{\Gamma}_{rel}(Q,F,W),
	}
\end{align}
where the dg functor $ J\circ\tilde{\delta} $ is a Morita functor.

\bigskip
	
Therefore,we have the following theorem which generalizes the result in~\cite[Theorem 3.2]{kellerDerivedEquivalencesMutations2011}. 
\begin{Thm}\label{Thm:equivalence non-frozen}
	Let $ \bm\Gamma_{rel}=\bm\Gamma_{rel}(Q,F,W) $ and $ \bm\Gamma'_{rel}=\bm\Gamma_{rel}(Q',F,W') $ be the complete Ginzburg dg algebras associated to $ (Q,F,W) $ and $ \tilde{\mu}_{v}(Q,F,W)=(Q',F,W'=W_{1}+W_{2})  $ respectively. For a vertex $ i $, let $ \bm\Gamma_{i}=e_{i}\bm\Gamma_{rel} $ and $ \bm\Gamma'_{i}=e_{i}\bm\Gamma'_{rel} $.
		\begin{itemize}
			\item[a)] There is a triangle equivalence
			\begin{align*}
				\xymatrix{
					\Phi_{+}:=J^{*}:\mathcal{D}(\bm\Gamma'_{rel})\ar[r]&\mathcal{D}(\bm\Gamma_{rel}),
				}
			\end{align*}
			which sends the the $ \bm\Gamma'_{j} $ to $ \bm\Gamma_{j} $ for $ j\neq v $ and to the cone over the morphism
			$$ \bm\Gamma_{v}\ra\bigoplus_{\alpha}\bm\Gamma_{t(\alpha)} $$ for $ j=v $, where we have a summand $ \bm\Gamma_{t(\alpha)} $ for each arrow $ \alpha $ of $ Q $ with source $ v $ and the corresponding component of the map is the left multiplication by $ \alpha $. The functor $ \Phi_{+} $ restricts to triangle equivalences from $ \per(\bm\Gamma'_{rel})$ to $ \per(\bm\Gamma_{rel}) $ and from $ \pvd(\bm\Gamma'_{rel}) $ to $ \pvd(\bm\Gamma_{rel}) $.
			\item[b)] Let $ \bm\Gamma^{red}_{rel} $ respectively $ \bm\Gamma'^{red}_{rel} $ be the complete Ginzburg dg algebra associated with the reduction of $ (Q,F,W) $ respectively the reduction $ \mu_{v}(Q,F,W)=(Q'',F'',W'') $ of $ \tilde{\mu}_{v}(Q,F,W) $. Then functor $ \Phi_{+} $ yields a triangle equivalence
			\begin{align*}
				\xymatrix{
					\Phi_{+}^{red}:\mathcal{D}(\bm\Gamma'^{red}_{rel})\ar[r]&\mathcal{D}(\bm\Gamma^{red}_{rel}),
				}
			\end{align*}
			which restricts to triangle equivalences from $ \per(\bm\Gamma'^{red}_{rel}) $ to $ \per(\bm\Gamma^{red}_{rel}) $ and from $ \pvd(\bm\Gamma'^{red}_{rel}) $ to $ \pvd(\bm\Gamma^{red}_{rel}) $.
			\item[c)] The following diagram commutes
			\[
			\begin{tikzcd}
				&&\cd(\bm\Gamma'_{rel})\arrow[dd,"\Phi_{+}"]\\
				\cd(\bm\Pi_{2}(F))\arrow[urr,"(\bm{G}'_{rel})^{*}"]\arrow[drr,swap,"(\bm{G}_{rel})^{*}"]&&\\
				&&\cd(\bm\Gamma_{rel}).
			\end{tikzcd}
			\]
			\item[d)] Since the frozen parts of $ \mu_{v}(Q,F,W)=(Q'',F'',W'') $ and of $ \tilde{\mu}_{v}(Q,F,W)=(Q',F,W') $ 
			only differ in the directions of the frozen arrows, we have a canonical isomorphism between $ \bm\Pi_{2}(kF) $ and $ \bm\Pi_{2}(kF'') $. It induces a canonical triangle equivalence
			$$ \mathrm{can}:\cd(\bm\Pi_{2}(F))\ra\cd(\bm\Pi_{2}(F'')) .$$ Moreover, the following diagram commutes up to isomorphism
			\[
			\begin{tikzcd}
				\cd(\bm\Pi_{2}(F''))\arrow[rr,"(\bm{G}'_{rel})^{*}"]\arrow[dd,"\mathrm{can}^{-1}",swap]&&\cd(\bm\Gamma'^{red}_{rel})\arrow[dd,"\Phi_{+}^{red}"]\\
				\\
				\cd(\bm\Pi_{2}(F))\arrow[rr,swap,"(\bm{G}_{rel})^{*}"]&&\cd(\bm\Gamma^{red}_{rel}).
			\end{tikzcd}
			\]
		\end{itemize}	
	\end{Thm}
	\begin{Rem}
		If $ v $ is the target of at least one arrow, we can use $ T'=\oplus_{j\neq v}P_{j}\oplus T'_{v} $, where $ T'_{v} $ is defined by the exact sequence
		\[
		\begin{tikzcd}
			0\arrow[r]&T'_{v}\arrow[r]&\displaystyle\bigoplus_{\beta\in Q_{1};t(\beta)=v}P_{s(\beta)}\arrow[r]&P_{v}\arrow[r]&0,
		\end{tikzcd}
		\]
		where the sum is taken over all arrows $ \beta:j\to v $ and the corresponding component of the map from $ P_{j} $ to $ P_{v} $ is the left multiplication by $ \beta $. Then $ T' $ is also a tilting object in $ \cd(\widehat{kQ}) $.
		
		There is also a triangle equivalence $ \Phi_{-}:\cd(\bm\Gamma'_{rel}) \ra\cd(\bm\Gamma_{rel}) $ which, for $ j\neq v $, sends the $ \bm\Gamma'_{j} $ to $ \bm\Gamma_{j} $ and for $ j=v $, to the shifted cone 
		$$ \Si^{-1}(\bigoplus_{\beta\in Q_{1};t(\beta)=v}\bm\Gamma_{s(\beta)}\ra \bm\Gamma_{v}),$$ where we have a summand $ \bm\Gamma_{s(\beta)} $ for each arrow $ \beta $ of $ Q $ with target $ i $ and the corresponding component of the morphism is left multiplication by $ \beta $. Moreover, the two equivalences $ \Phi_{+} $ and $ \Phi_{-} $ are related by the twist functor $ t_{S_{v}} $ with with respect to the 3-spherical object $ S_{v} $, i.e. $ \Phi_{-}=t_{S_{v}}\circ\Phi_{+} $. For each object $ X $ in $ \cd(\bm{\Gamma}_{rel}) $, the object $ t_{S_{v}}(X) $ is given by the following triangle
		$$ \RHom(S_{v},X)\ten_{k}S_{v}\ra X\ra t_{S_{v}}(X)\ra\Si\RHom(S_{v},X)\ten_{k}S_{v} .$$
		
	\end{Rem}
	
	\begin{Rem}
		For an oriented marked surface with ideal triangulation $ (\mathbf{S},\ct) $, we denote by $ \cg_{\ct} $ the associated relative Ginzburg algebra (see~\cite{christGinzburgAlgebrasTriangulated2021}). Let $ \mathbf{S} $ be an oriented marked surface with two ideal triangulation $ \ct $, $ \ct' $ related by a flip of an edge $ e $ of $ \ct $. Recently, Merlin Christ~\cite[Theorem 3]{christGinzburgAlgebrasTriangulated2021} has proved that there exists an equivalence of stable $ \infty $-categories
		$$ \mu_{e}:\cd(\cg_{\ct})\simeq\cd(\cg_{\ct'}) .$$ 
		One can extend the results of~\cite{labardini-fragosoQuiversPotentialsAssociated2009} relating flips of ideal triangulations and mutations of quivers
		with potential to the case of ice quivers with potential. In the absence of punctures, we may then give an alternative approach to Christ's result by using Theorem~\ref{Thm:equivalence non-frozen}.
	\end{Rem}
	
	\begin{Def}\rm\label{Def:boundary dg algebra}
		Let $ (Q,F,W) $ be a ice quiver with potential. The \emph{boundary dg algebra} is defined to be the dg subalgebra of $ \bm{\Gamma}_{rel}(Q,F,W) $ $$ \Bd(Q,F,W)=\mathrm{REnd}_{\bm\Gamma_{rel}(Q,F,W)}(\bm{G}_{rel}^{*}(\bm{\Pi}_{2}(F)))\simeq e_{F}\bm\Gamma_{rel}(Q,F,W)e_{F} ,$$ where $ e_{F}=\sum_{i\in F}e_{i} $ is the sum of idempotents corresponding to the frozen vertices.
	\end{Def}
	
	\begin{Cor}\rm\label{Cor:The boundary dg algebra is invariant under mutations}
		The boundary dg algebra is invariant under the mutations at unfrozen vertices. Moreover, if $ \bm\Gamma_{rel}(Q,F,W) $ is concentrated in degree 0, then the boundary dg algebra $ \Bd(Q,F,W) $ is also concentrated in degree 0.
	\end{Cor}
	\begin{proof}
		This follows from the Definition of boundary dg algebra and part (c) of Theorem~\ref{Thm:equivalence non-frozen}.
		
	\end{proof}
	
	\subsection{Stability under mutation of relative Ginzburg algebras concentrated in degree 0}
	Let $ (Q,F,W) $ be an ice quiver with potential. Let $ \mathbf{\Gamma}_{rel} $ be the complete Ginzburg algebra associated with $ (Q,F,W) $. For each vertex $ i $ of $ Q $, we denote by $ \mathbf{\Gamma}_{i} $ the cofibrant dg $ \mathbf{\Gamma}_{rel} $-module associated with $ i $. Recall that the relative Jacobian algebra $ J_{rel}=J(Q,F,W) $ is the 0-th homology of $ \mathbf{\Gamma}_{rel}(Q,F,W) $.
	
	For each vertex $ i $ of $ Q $, we denote by $ S_{i} $ the associated simple module and by $ P_{i}=e_{i}J_{rel} $ its projective cover. Let $ v $ be an unfrozen vertex. Consider the complex $ T'_{rel} $ which is the sum of the $ P_{j} $, $ j \neq v $, concentrated in degree 0 and of the complex
	$$ 0\ra P_{v}\xrightarrow{c}\bigoplus_{\alpha\in Q_{1}:s(\alpha)=v}P_{t(\alpha)}\ra0 ,$$ where $ P_{v} $ is in degree -1 and the components of $ c $ are the left multiplications by the corresponding arrows.
	
	\begin{Thm}\cite[Theorem 6.2]{kellerDerivedEquivalencesMutations2011}
		Suppose that the complete Ginzburg algebra $ \bm\Gamma_{rel}=\bm\Gamma_{rel}(Q,F,W) $ has its homology concentrated in degree $ 0 $. Then $ T'_{rel} $ is a tilting object in the perfect derived category $ \per(J_{rel})\simeq\per(\bm\Gamma_{rel}) $. Thus, the complete relative Ginzburg algebra $ \bm\Gamma'_{rel} $ associated with $ \mu_{v}(Q,F,W) $ still has its homology concentrated in degree 0 and then $ \Bd(\mu_{v}(Q,F,W)) $ is concentrated in degree 0.
		
	\end{Thm}
	\begin{proof}
		The proof follows the lines of that of~\cite[Theorem 6.2]{kellerDerivedEquivalencesMutations2011}. There is a decomposition of $ J_{rel} $ as right $ J_{rel} $-module
		$$ J_{rel}=P_{v}\oplus\bigoplus_{i\in Q_{0}:i\neq v}P_{i} .$$
		By the construction of $ T'_{rel} $, we have a map $ c:P_{v}\ra\bigoplus_{\alpha\in Q_{1}:s(\alpha)=v}P_{t(\alpha)} $. We set $ B=\bigoplus_{\alpha\in Q_{1};s(\alpha)=v}P_{t(\alpha)} $ and $ T_{1}=\bigoplus_{i\in Q_{0}:i\neq v}P_{i} $. Then by~\cite[Proposition 6.5]{kellerDerivedEquivalencesMutations2011}, we have to check that the map $ c:P_{v}\ra B $ satisfies the following conditions
		\begin{itemize}
			\item[1)] $ B $ belongs to $ \add(T_{1}) $;
			\item[2)] the map $ c^{*}:\per(\bm\Gamma_{rel})(B,T_{1})\ra\per(\bm\Gamma_{rel})(P_{v},T_{1}) $ is surjective and
			\item[3)] the map $ c_{*}:\per(\bm\Gamma_{rel})(T_{1},P_{v})\ra\per(\bm\Gamma_{rel})(T_{1},B) $ is injective.
		\end{itemize}
		Condition 1) holds since $ B $ belongs to $ \add(T_{1}) $. Condition 2) holds since $ c: P_{v}\ra B $ is a left $ \add(T_{1}) $-approximation. Finally, in order to show condition 3), it is enough to show that $ c $ is injective. Since the homology of $ \Gamma_{rel} $ is concentrated in degree 0, the functor $ ?\lten_{\bm\Gamma_{rel}}J_{rel} $ is an equivalence from $ \per(\bm\Gamma_{rel}) $ to $ \per(J_{rel}) $ whose inverse is given by the restriction along the projection morphism $ \bm\Gamma_{rel}\ra J_{rel} $.
		
		By Subsection~\ref{subsection:Cofibrant resolutions of simples over a tensor algebra}, we have the following exact sequence in $ \cc(\bm\Gamma_{rel}) $
		\[
		\begin{tikzcd}
			0\arrow[r]&\ker(\pi)\arrow[r]&\bm\Gamma_{v}\arrow[r]&S_{v}\arrow[r]&0,
		\end{tikzcd}
		\]
		where $ \pi $ is the canonical projection from $ \bm\Gamma_{v} $ to $ S_{v} $. By the definition of relative Ginzburg algebra (see Definition~\ref{Def:Relative Ginzburg algebra}), the graded $ \bm\Gamma_{rel} $-module $ \ker(\pi) $ decomposes as
		\begin{equation*}
			\begin{split}
					\ker(\pi)=&\bigoplus_{\rho\in Q_{1}:t(\rho)=v}\rho \bm{\Gamma}_{s(\rho)}\\
					=&\bigoplus_{a\in Q_{1}:t(a)=v}a \bm{\Gamma}_{s(a)}\oplus\bigoplus_{b\in Q_{1}:s(b)=v}b^{\vee} \bm{\Gamma}_{t(b)}\oplus t_{v}\bm{\Gamma}_{v}
			\end{split}
		\end{equation*} 
		with the induced differential. Let $ P_{S_{v}} $ be the mapping cone of $ \ker(\pi)\ra\bm\Gamma_{v} $. Then the canonical map $ P_{S_{v}}\ra S_{v} $ is a cofibrant replacement of $ S_{v} $.
		
		Applying the equivalence $ ?\lten_{\bm\Gamma_{rel}}J_{rel} $ to $ P_{S_{v}}\ra S_{v} $, we obtain a projective resolution of $ S_{v} $ in the category of $ J_{rel} $-modules
		$$ 0\ra P_{v}\xrightarrow{c} B\ra B'\ra P_{v}\ra S_{v}\ra0 .$$
		Thus, the map $ c $ is injective.
	\end{proof}

	\bigskip
	\bigskip
	Let $ Q_{0}^{m}=Q_{0}\setminus F_{0} $ and $ Q_{1}^{m}=Q_{1}\setminus F_{1} $. Let $ S $ be the semisimple $ k $-algebra $ \prod_{i\in Q_{0}}ke_{i} $. We denote  by $ R^{m} $, $ V $ and $ V^{m} $ the $ S $-bimodules generated by $ Q_{0}^{m} $, $ Q_{1} $ and $ Q_{1}^{m} $ respectively. Let $ V^{m*} $ be the dual bimodule $ \Hom_{S^{e}}(V^{m},S^{e}) $.
	
	We have a canonical short exact sequence of $ \mathbf{\Gamma}_{rel} $-bimodules
	$$ 0\ra\ker(m)\xrightarrow{\rho}\mathbf{\Gamma}_{rel}\ten_{S}\mathbf{\Gamma}_{rel}\xrightarrow{\delta}\mathbf{\Gamma}_{rel}\ra0 ,$$ where the map $ \delta $ is induced by the multiplication of $ \mathbf{\Gamma}_{rel} $, $ \ker(m)=\mathbf{\Gamma}_{rel}\ten_{S}(V\oplus\Si V^{m*}\oplus\Si^{2}R^{m})\ten_{S}\mathbf{\Gamma}_{rel} $ and $ \rho $ maps $ \alpha\ten f\ten\beta $ to $ \alpha f\ten\beta-\alpha\ten f\beta $. The mapping cone $ \cone(\rho) $ of $ \rho $ is a cofibrant resolution of $ \mathbf{\Gamma}_{rel} $ as a bimodule over itself.
	
	Then $ J_{rel}\otimes_{\bm\Gamma_{rel}}\cone(\rho)\otimes_{\bm\Gamma_{rel}}J_{rel} $ is the following complex and we denote it by $ P(J_{rel}) $ (see~\cite[pp. 10]{presslandCalabiYauPropertiesPostnikov2019})
	\begin{align*}
		\xymatrix{
			0\ar[r]&J_{rel}\otimes_{S} R^{m}\otimes_{S}J_{rel}\ar[r]^{m_{3}}&J_{rel}\otimes_{S} V^{m*}\otimes_{S}J_{rel}\ar[r]^{m_{2}}&J_{rel}\otimes_{S} V\otimes_{S}J_{rel}\ar[r]^-{m_{1}}&J_{rel}\otimes_{S}J_{rel}\ar[r]&0,
		}
	\end{align*} 
	where $ m_{3} $, $ m_{2} $ and $ m_{1} $ are given by as follows:
	$$ m_{1}(x\otimes a\otimes y)=xa\otimes y-x\otimes ay $$ and  
	$$ m_{3}(x\otimes t_{i}\otimes y)=\sum_{a,t(a)=t_{i}}xa\otimes a^{*}\otimes y-\sum_{b,s(b)=t_{i}}x\otimes b^{*}\otimes by. $$
	
	For any path $ p=a_{m}\cdots a_{1} $ of $ kQ $, we define
	\begin{align*}
		\triangle_{a}(p)=\sum_{a_{i}=a}a_{m\cdots}a_{i+1}\otimes a_{i}\otimes a_{i-1}\cdots a_{1},
	\end{align*}
	and extend by linearity to obtain a map $ \triangle_{a}:kQ\to J_{rel}\otimes_{S} kQ_{1}\otimes_{S} J_{rel} $. Then $ m_{2} $ is given by:
	$$ m_{2}(x\otimes a^{*}\otimes y)=\sum_{b\in Q_{1}}x\triangle_{b}(\partial_{a}W)y.$$
	
	There is a canonical morphism $ P(J_{rel}) \to J_{rel} $, which is induced by the multiplication map $ m $ in $ J_{rel} $
	\begin{align}\label{extended cotangent complex}
		\xymatrix{
			0\ar[r]&J_{rel}\otimes_{S} R^{m}\otimes_{S}J_{rel}\ar[r]^{m_{3}}\ar[d]&J_{rel}\otimes_{S} V^{m*}\otimes_{S}J_{rel}\ar[r]^{m_{2}}\ar[d]&J_{rel}\otimes_{S} V\otimes_{S}J_{rel}\ar[r]^-{m_{1}}\ar[d]&J_{rel}\otimes_{S}J_{rel}\ar[r]\ar[d]^{m}&0
			\\
			&0\ar[r]&0\ar[r]&0\ar[r]&J_{rel}.
		}
	\end{align} 
	
	\begin{Rem}
		When $ F=\emptyset $, the complex~(\ref{extended cotangent complex}) defined above is the complex associated to $ (Q, W) $ defined by Ginzburg in~\cite[Section 5]{ginzburgCalabiyauAlgebras2006a}. In general, it is exactly the complex $ P(J_{rel}) $ defined by Pressland in~\cite[pp. 10]{presslandCalabiYauPropertiesPostnikov2019}. Moreover, it has already appeared in work of Amiot–Reiten–Todorov (see~\cite[Propostion 2.2]{amiotUbiquityGeneralizedCluster2011a}). 
	\end{Rem}

	\begin{Lem}\label{Lemma:concentrated in degree 0}
		If complex (\ref{extended cotangent complex}) is exact, then $ \mathbf{\Gamma}_{rel} $ is concentrated in degree 0.
	\end{Lem}
	\begin{proof}
		Let $ \cd_{pc}(\mathbf{\Gamma}_{rel}) $ be the pseudocompact derived category of $ \mathbf{\Gamma}_{rel} $. For each vertex $ i $, we denote by $ S_{i} $ the simple $  \mathbf{\Gamma}_{rel} $-module (or $ J_{rel} $-module) associated with $ i $. By Proposition~\ref{Prop:pseudocompact for smooth dg}, the opposite category $ \cd_{pc}(\mathbf{\Gamma}_{rel})^{op} $ is compactly generated by $ \{S_{i}\,|\,i\in Q_{0}\} $ and similarly for $ \cd_{pc}(J_{rel})^{op} $. The restriction functor $$ R:\cd_{pc}(J_{rel})\ra\cd_{pc}(\mathbf{\Gamma}_{rel}) $$ takes $ S_{i} $ to $ S_{i} $. Thus, we can conclude that $ R $ is an equivalence if it induces isomorphisms
		$$ \Ext_{J_{rel}}^{*}(S_{i},S_{j})\iso\Ext_{\mathbf{\Gamma}_{rel}}^{*}(S_{i},S_{j}) ,\,\forall i,j\in Q_{0} .$$
		
		If the complex (\ref{extended cotangent complex}) is exact, then $ P(J_{rel}) $ is a projective resolution of $ J_{rel} $ as a bimodule over itself. Thus, for each vertex $ i\in Q_{0} $, $ S_{i}\ten_{J_{rel}}P(J_{rel}) $ is a cofibrant resolution of $ S_{i} $ as a right $ J_{rel} $-module. So we have:
		\begin{equation*}
			\begin{split}
				\RHom_{J_{rel}}(S_{i},S_{j})&=\Hom_{J_{rel}}(S_{i}\ten_{J_{rel}}P(J_{rel}),S_{j})\\
				&=\Hom_{J_{rel}}(S_{i}\ten_{\mathbf{\Gamma}_{rel}}\cone(\rho)\ten_{\mathbf{\Gamma}_{rel}}J_{rel},S_{j})\\
				&=\Hom_{\mathbf{\Gamma}_{rel}}(S_{i}\ten_{\mathbf{\Gamma}_{rel}}\cone(\rho),\Hom_{J_{rel}}(J_{rel},S_{j}))\\
				&=\Hom_{\mathbf{\Gamma}_{rel}}(S_{i}\ten_{\mathbf{\Gamma}_{rel}}\cone(\rho),S_{j})\\
				&=\RHom_{\mathbf{\Gamma}_{rel}}(S_{i},S_{j}).
			\end{split}
		\end{equation*}
		Thus, the restriction functor $ R $ is an equivalence. It follows that $ \mathbf{\Gamma}_{rel} $ is concentrated in degree $ 0 $.
	\end{proof}
	
	\begin{Ex}\label{Ex:Postnikov diagram}
		Let $ D $ be a Postnikov diagram in the disc (see~\cite{presslandCalabiYauPropertiesPostnikov2019}). We can associate to $ D $ an ice quiver with potential $ (Q_{D},F_{D},W_{D}) $ (see~\cite[Definition 2.3]{presslandCalabiYauPropertiesPostnikov2019}). By~\cite[Proposition 3.6]{presslandCalabiYauPropertiesPostnikov2019} and Lemma~\ref{Lemma:concentrated in degree 0}, the corresponding complete relative Ginzburg algebra $ \bm{\Gamma}_{rel}(Q_{D},F_{D},W_{D}) $ is concentrated in degree 0. Thus, the associated boundary dg algebra $ \Bd(Q_{D},F_{D},W_{D}) $ is also concentrated in degree 0. Hence the boundary dg algebra is invariant under the mutations at the unfrozen vertices. 
		
		If $ D $ has the property that every strand has exactly $ k $ boundary regions on its right, then each strand must terminate at a marked point $ k $ steps clockwise from its source. Such $ D $ is called a \emph{(k,n)-diagram} (see~\cite{baurDimerModelsCluster2016}). In this case, Corollary~\ref{Cor:The boundary dg algebra is invariant under mutations} gives a different proof of Baur--King--Marsh's result~\cite[Corollary 10.4]{baurDimerModelsCluster2016} which says that the boundary algebra is independent of the choice of Postnikov diagram $ D $, up to isomorphism.
		
	\end{Ex}

	\section{Mutation at frozen vertices}\label{Section6}
	Let $ (Q,F) $ be an ice quiver. Let $ v $ be a frozen vertex.
	\begin{Def}\rm
		We say that $ v $ is a \emph{frozen source} of $ Q $ if $ v $ is a source vertex of $ F $ and no unfrozen arrows with source $ v $. Similarly, We say that $ v $ is a \emph{frozen sink} of $ Q $ if $ v $ is a sink vertex of $ F $ and no unfrozen arrows with target $ v $. For two vertices $ i $ and $ j $, we say that they have the \emph{same state} if they are both in $ F_{0} $ or $ Q_{0}\setminus F_{0} $. Otherwise, we say that they have \emph{different state}. Similarly, for two arrows in $ Q $, we say that they have the \emph{same state} if they are both in $ F_{1} $ or $ Q_{1}\setminus F_{1} $. Otherwise, we say that they have \emph{different state}.
	\end{Def}

	\subsection{Combinatorial mutations}
	Mutation at frozen vertices first appears in recent work of Fraser--Sherman-Bennett on positroid cluster structures~\cite{fraserPositroidClusterStructures2020}.
	\begin{Def}\rm\label{Def:Combinatorial ice mutations}
		Let $ v $ be a frozen source or a frozen sink of $ Q $ such that no loops or 2-cycles of $ Q $ are incident with $ v $. The \emph{mutation} $ \mu_{v}^{P}(Q,F)=(\mu_{v}^{P}(Q),\mu_{v}^{P}(F)) $ of $ (Q, F) $ at $ v $ is defined to be the output of the following procedure.
		\begin{itemize}
			\item[(1)] For each pair of arrows $ \alpha:u\ra v$ and $ \beta:v\ra w $, add an unfrozen arrow $ [\beta\alpha]:u\ra w $ to $ Q $.
			\item[(2)] Replace each arrow $ \alpha:u\ra v $ by an arrow $ \alpha^{*}:v\ra u $ of the same state as $ \alpha $ and each arrow $ \beta:v\ra w $ by an arrow $ \beta^{*}:w\ra v $ of the same state as $ \beta $.
			\item[(3)] Remove a maximal collection of unfrozen 2-cycles, i.e. 2-cycles avoiding the subquiver $ F $.
			\item[(4)] Choose a maximal collection of half-frozen 2-cycles, i.e. 2-cycles in which precisely one arrow is frozen. Replace each 2-cycle in this collection by a frozen arrow, in the direction of the unfrozen arrow in the 2-cycle.
		\end{itemize}
	\end{Def}
	
	\begin{Rem}
		The procedure in the Definition above is the same as in Pressland's Definition~\ref{Def:Combinatorial mutations}.
	\end{Rem}
	
	\begin{Ex}\label{Example:mutation at frozen}
		Consider the following ice quiver $ (Q,F) $
		\[
		\begin{tikzcd}[row sep=1ex]
			&\color{blue}\boxed{1}\arrow[ddr]&&\color{blue}\boxed{2}\arrow[ll,blue]\arrow[dddd,blue]\\
			\ \\
			(Q,F)=&&5\arrow[uur]\arrow[ddl]&\\
			\ \\
			&\color{blue}\boxed{3}\arrow[uuuu,blue]\arrow[rr,blue]&&\color{blue}\boxed{4}\arrow[uul]\,,
		\end{tikzcd}
		\] where the frozen subquiver $ F $ is drawn in blue. The frozen vertex 3 is a frozen source. Performing the mutation at vertex 3, we get the following ice quiver
		
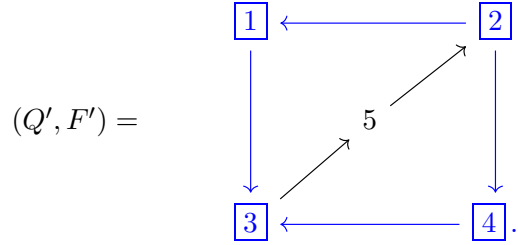
\[
		\begin{tikzcd}[row sep=1ex]
			&\color{blue}\boxed{1}\arrow[dddd,blue]&&\color{blue}\boxed{2}\arrow[ll,blue]\arrow[dddd,blue]\\
			\ \\
			(Q',F')=&&5\arrow[uur]&\\
			\ \\
			&\color{blue}\boxed{3}\arrow[uur]&&\color{blue}\boxed{4}\arrow[ll,blue]\,.
		\end{tikzcd}
		\]
		The frozen vertex 2 is a frozen source. Performing the mutation at vertex 2, we get the following ice quiver
		
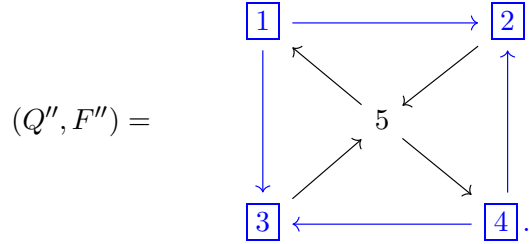
\[
		\begin{tikzcd}[row sep=1ex]
			&\color{blue}\boxed{1}\arrow[dddd,blue]\arrow[rr,blue]&&\color{blue}\boxed{2}\arrow[ddl]\\
			\ \\
			(Q'',F'')=&&5\arrow[uul]\arrow[ddr]&\\
			\ \\
			&\color{blue}\boxed{3}\arrow[uur]&&\color{blue}\boxed{4}\arrow[ll,blue]\arrow[uuuu,blue]\,.
		\end{tikzcd}
		\]
		
		Surprisingly, after performing the mutation $ \mu^{P}_{5}(Q,F) $ at the unfrozen vertex 5, we get the same ice quiver $ (Q'',F'') $, i.e. $ \mu^{P}_{5}(Q,F)=\mu^{P}_{2}(\mu^{P}_{3}(Q,F)) $. 
		
	\end{Ex}
	
	\subsection{Algebraic mutations}
	Let $ (Q,F,W) $ be an ice quiver with an irredundant potential. Let $ v $ be a frozen source or a frozen sink such that no loops or 2-cycles of $ Q $ are incident with $ v $.
	\begin{Def}\rm\cite[Definition 4.1]{presslandMutationFrozenJacobian2020}\label{Def:algebraic ice mutation}
		The ice quiver with potential $ \tilde{\mu}_{v}(Q,F,W) $, called the \emph{pre-mutation} of $ (Q,F,W) $ at $ v $, is the output of the following procedure.
		\begin{itemize}
			\item[(1)] For each pair of arrows $ \alpha:u\to v $ and $ \beta:v\to w $, add an unfrozen ‘composite' arrow $ [\beta\alpha]:u\to w $ to $ Q $.
			\item[(2)] Replace each arrow $ \alpha:u\to v $ by an arrow by an arrow $ \alpha^{*}:v\to u $ of the same state as $ \alpha $ and each arrow $ \beta:v\to w $ by an arrow $ \beta^{*}:w\to v $ of the same state as $ \beta $.
			\item[(3)] Pick a representative $ \widetilde{W} $ of $ W $ in $ kQ $ such that no term of $ W $ begins at $ v $ (which is possible since there are no loops at $ v $). For each pair of arrows $ \alpha,\beta $ as in $ (1) $, replace each occurrence of $ \beta\alpha $ in $ \widetilde{W} $ by $ [\beta\alpha] $, and add the term $ [\beta\alpha]\alpha^{*}\beta^{*} $.
		\end{itemize}
		Let us write $ (Q',F',W') $ for $ \tilde{\mu}_{v}(Q,F,W) $. It is clear that the new potential $ W' $ is also irredundant, since the arrows $ [\beta\alpha] $ are unfrozen, but it need not be reduced even if $ (Q,F,W) $ is. We define $ \mu_{v}(Q,F,W) $ by replacing the resulting ice quiver with potential $ \tilde{\mu}_{v}(Q,F,W) $ by its reduction, as in Theorem~\ref{Thm:reduction}, this being unique up to right equivalence by Proposition~\ref{Prop:rduction up to right equ}. We call $ \mu_{v} $ the \emph{mutation} at the vertex $ v $. 
	\end{Def}
	
	\begin{Ex}
		Consider the following ice quiver $ (Q,F) $
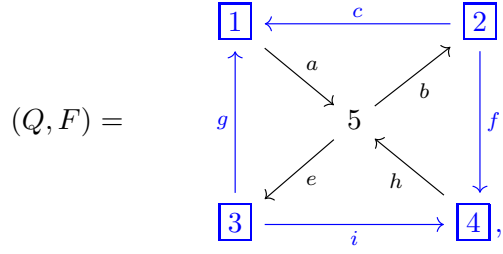
		\[
		\begin{tikzcd}[row sep=1ex]
			&\color{blue}\boxed{1}\arrow[ddr,"a"]&&\color{blue}\boxed{2}\arrow[ll,blue,swap,"c"]\arrow[dddd,blue,"f"]\\
			\ \\
			(Q,F)=&&5\arrow[uur,swap,"b"]\arrow[ddl,"e"]&\\
			\ \\
			&\color{blue}\boxed{3}\arrow[uuuu,blue,"g"]\arrow[rr,blue,swap,"i"]&&\color{blue}\boxed{4}\arrow[uul,"h"]\,,
		\end{tikzcd}
		\] where the blue subquiver represents the frozen subquiver $ F $. We consider the potential $$ W=cba-gea+hie-fbh .$$
		The vertex 3 is a frozen source. The pre-mutation at vertex 3 is the following ice quiver with potential
		
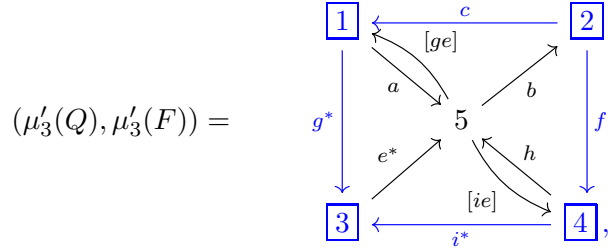
\[
		\begin{tikzcd}[row sep=1ex]
			&\color{blue}\boxed{1}\arrow[ddr,swap,"a"]\arrow[dddd,blue,swap,"g^{*}"]&&\color{blue}\boxed{2}\arrow[ll,blue,swap,"c"]\arrow[dddd,blue,"f"]\\
			\ \\
			(\mu'_{3}(Q),\mu'_{3}(F))=&&5\arrow[uur,swap,"b"]\arrow[bend right=20,swap]{uul}{[ge]}\arrow[bend right=20,swap]{ddr}{[ie]}&\\
			\ \\
			&\color{blue}\boxed{3}\arrow[uur,"e^{*}"]&&\color{blue}\boxed{4}\arrow[uul,swap,"h"]\arrow[ll,blue,"i^{*}"]\,,
		\end{tikzcd}
		\]
		where the blue subquiver represents the frozen subquiver $ \mu'_{3}(F) $.
		The new potential is given by $ \mu'_{3}(W) $ is given by 
		$$ \mu'_{3}(W)=cba-[ge]a+h[ie]-fbh+[ge]e^{*}g^{*}+[ie]e^{*}i^{*}.$$
		This ice quiver with potential is not reduced. Then $ \mu_{3}(Q,F,W) $ is given by its reduction, which is the following ice quiver with potential
		
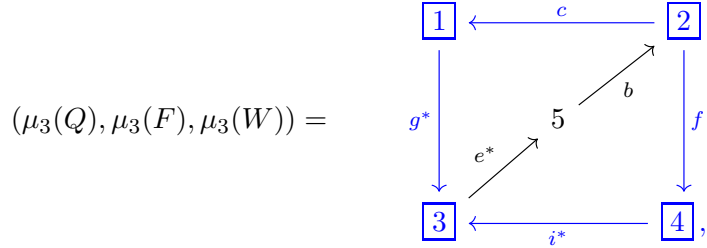
\[
		\begin{tikzcd}[row sep=1ex]
			&\color{blue}\boxed{1}\arrow[dddd,blue,swap,"g^{*}"]&&\color{blue}\boxed{2}\arrow[ll,blue,swap,"c"]\arrow[dddd,blue,"f"]\\
			\ \\
			(\mu_{3}(Q),\mu_{3}(F),\mu_{3}(W))=&&5\arrow[uur,swap,"b"]&\\
			\ \\
			&\color{blue}\boxed{3}\arrow[uur,"e^{*}"]&&\color{blue}\boxed{4}\arrow[ll,blue,"i^{*}"]\,,
		\end{tikzcd}
		\]
		where the blue subquiver represents the frozen subquiver $ \mu_{3}(F) $ and the new potential $ \mu_{3}(W) $ is zero. We see that the underlying ice quiver of $ \mu_{3}(Q,F,W) $ is the same as $ \mu^{P}_{3}(Q,F) $ in Example~\ref{Example:mutation at frozen}.
	\end{Ex}
	
	\begin{Thm}\cite[Proposition 4.6]{presslandMutationFrozenJacobian2020}
		Let $ (Q,F,W) $ be an ice quiver with potential and $ v $ a frozen source or a frozen sink. If $ (\mu_{v}Q,\mu_{v}F) $ has no $ 2 $-cycles containing unfrozen arrows, then the underlying ice quiver of $ \mu_{v}(Q,F,W) $ agrees with $ \mu^{P}_{v}(Q,F) $ defined in Definition~\ref{Def:Combinatorial ice mutations}.
	\end{Thm}

	\subsection{Categorical mutations}
	Let $ (Q,F,W) $ be an ice quiver with potential. Let $ v $ be a frozen source. Write $ (Q',F',W')=\tilde{\mu}_{v}(Q,F,W) $. Let $ \bm\Gamma_{rel}=\bm\Gamma_{rel}(Q,F,W) $ and $ \bm\Gamma'_{rel}=\bm\Gamma_{rel}(Q',F',W') $ be the complete relative Ginzburg dg algebras associated to $ (Q,F,W) $ and $ (Q',F',W') $ respectively. For a vertex $ i $, let $ \bm\Gamma_{i}=e_{i}\bm\Gamma_{rel} $ and $ \bm\Gamma'_{i}=e_{i}\bm\Gamma'_{rel} $.
	\begin{Thm}\label{Thm:categorification of ice mutaion1}
		We have a triangle equivalence
		$$ \Psi_{+}:\cd(\bm\Gamma'_{rel})\ra\cd(\bm\Gamma_{rel}) ,$$
		which sends the $ \bm\Gamma'_{i} $ to $ \mathbf\Gamma_{i} $ for $ i\neq v $ and $ \mathbf\Gamma_{v} $ to the cone
		$$ \cone(\bm\Gamma_{v}\ra\bigoplus_{\alpha}\bm\Gamma_{t(\alpha)}),$$ where we have a summand $ \bm\Gamma_{t(\alpha)} $ for each arrow $ \alpha $ of $ F $ with source $ v $ and the corresponding component of the map is the left multiplication by $ \alpha $. The functor $ \Psi_{+} $ restricts to triangle equivalences from $ \per(\bm\Gamma'_{rel}) $ to $ \per(\bm\Gamma_{rel}) $ and from $ \pvd(\bm\Gamma'_{rel}) $ to $ \pvd(\bm\Gamma_{rel}) $. Moreover, the following square commutes up to isomorphism
		\begin{equation}\label{Diagram:left ice mutation}
			\begin{tikzcd}
				\cd(\bm\Pi_{2}(F'))\arrow[r]\arrow[d,swap,"\mathrm{can}"]&\cd(\bm\Gamma'_{rel})\arrow[d,"\Psi_{+}"]\\
				\cd(\bm\Pi_{2}(F))\arrow[d,swap,"t^{-1}_{S_{v}}"]&\cd(\bm\Gamma_{rel})\arrow[d,equal]\\
				\cd(\bm\Pi_{2}(F))\arrow[r,swap]&\cd(\bm\Gamma_{rel}),
			\end{tikzcd}
		\end{equation}
		where $\mathrm{can}$ is the canonical functor induced by an identification between $ \bm\Pi_{2}(F') $ and $ \bm\Pi_{2}(F) $ and $ t^{-1}_{S_{v}} $ is the inverse twist functor with respect to the $ 2 $-spherical object $ S_{v} $, which gives rise to a triangle
		$$ t^{-1}_{S_{v}}(X)\ra X\ra\Hom_{k}(\RHom_{\bm\Pi_{2}(F)}(X,S_{v}),S_{v})\ra\Si t^{-1}_{S_{v}}(X) $$ for each object $ X $ of $ \cd(\bm\Pi_{2}(F)) $.
	\end{Thm}
	\begin{proof}
		Let $ P_{i} $ be the projective $ \widehat{kQ} $-module associated with vertex $ i $. Let $ M=\displaystyle\bigoplus_{\substack{j\neq v}}P_{j}\oplus M_{v}$, where $ M_{v} $ is defined by the exact sequence
		\[
		\begin{tikzcd}
			0\arrow[r]& P_{v}\arrow[r]&\displaystyle\bigoplus_{\alpha\in F_{1}:s(\alpha)=v}P_{j}\arrow[r]& M_{v}\arrow[r]&0,
		\end{tikzcd}
		\]
		where the sum is taken over all frozen arrows $ \alpha:v\to j $ and the corresponding component of the map from $ P_{v} $ to the sum is the left multiplication by $ \alpha $. Since $ v $ is a frozen source, we only have frozen arrows with source $ v $. It is easy to see that $ M $ is a tilting object of $ \cd(\widehat{kQ}) $. By abuse of notation, let $ A' $ be the endomorphism algebra $ \End_{\widehat{kQ}}(M) $.
		
		Let $ \ca $ be the full dg subcategory of $ \cc_{pc}^{dg}(\widehat{kQ}) $ whose objects are $ M_{v} $ and the $ P_{i} $, $ i\neq v $. Then $ \ca $ is Morita equivalent to $ \widehat{kQ} $. Let $ \ca' $ be the full dg subcategory of $ \cc_{pc}^{dg}(A') $ whose objects are the $ P'_{i}=e_{i}A' $, $ i\in Q_{0} $. Then $ \ca' $ is equal to $ A' $. We define a dg functor
		$$ J:\ca'\ra\ca $$ as follows:
		\begin{itemize}
			\item For $ k\neq v $, we put $ J(P'_{k})=P_{k} $,
			\item For $ k=v $, we put $ J(P'_{v})=M_{v} $. 
		\end{itemize}
		Then $ J $ is a Morita functor. It induces an isomorphism $ HC_{0}(A')\simeq HC_{0}(\widehat{kQ}) $. Let $ W'\in HC_{0}(A') $ be the element corresponding to $ W\in HC_{0}(\widehat{kQ}) $. We set $ \xi=B(W)\in HH_{1}(\widehat{kQ}) $ and $ \xi'=B(W')\in HH_{1}(A') $, where $ B $ is the Connes connecting map.
		
		Similarly, let $ N=\displaystyle\bigoplus_{\substack{j\neq v}}e_{j}\widehat{kF}\oplus N_{v} $, where $ N_{v} $ is defined by the exact sequence
		\[
		\begin{tikzcd}
			0\arrow[r]& e_{v}\widehat{kF}\arrow[r]&\displaystyle\bigoplus_{\alpha\in F_{1}:s(\alpha)=v}e_{j}\widehat{kF}\arrow[r]& N_{v}\arrow[r]&0,
		\end{tikzcd}
		\]
		where the sum is taken over all frozen arrows $ \alpha:v\to j $ and the corresponding component of the map from $ e_{v}\widehat{kF} $ to the sum is the left multiplication by $ \alpha $. It is easy to see that $ N $ is a tilting object of $ \cd(\widehat{kF}) $. Let $ B' $ be the endomorphism algebra $ \End_{\widehat{kF}}(N) $.
		
		Let $ \cb $ be the full dg subcategory of $ \cc_{pc}^{dg}(\widehat{kF}) $ whose objects are $ N_{v} $ and the $ e_{i}\widehat{kF} $, $ i\neq v $. Then $ \cb $ is Morita equivalent to $ \widehat{kF} $. Let $ \cb' $ be the full dg subcategory of $ \cc_{pc}^{dg}(B') $ whose objects are the $ e_{i}B' $, $ i\in F_{0} $. Then $ \cb' $ is equal to $ B' $. We define a dg functor
		$$ I:\cb'\ra\cb $$ as follows:
		\begin{itemize}
			\item For $ k\neq v $, we put $ I(e_{k}B')=e_{k}\widehat{kF} $,
			\item For $ k=v $, we put $ I(e_{v}B')=N_{v} $. 
		\end{itemize}
		Then $ I $ is a Morita functor. And we have the following commutative diagram 
		\[
		\begin{tikzcd}
			\cb'\arrow[r,"i'"]\arrow[d,"I"]&\ca'\arrow[d,"J"]\\
			\cb\arrow[r,"i"]&\ca,
		\end{tikzcd}
		\]
		where $ i' $ maps $ e_{i}B' $ to $ e_{i}A' $ and $ i $ sends $ e_{i}\widehat{kF} $ to $ e_{i}\widehat{kQ} $ for $ i\neq v $ and $ N_{v} $ to $ M_{v} $.

		By Theorem~\ref{Thm:Derived equivalence}, we have the following commutative diagram
		\[
		\begin{tikzcd}
			\bm{\Pi}_{2}(\cb')\arrow[r,"\tilde{i'}"]\arrow[d,"I'"]&\bm\Pi_{3}(\ca',\cb',\xi')\arrow[d,"J'"]\\
			\bm{\Pi}_{2}(\cb)\arrow[r,"\tilde{i}"]&\bm\Pi_{3}(\ca,\cb,\xi),
		\end{tikzcd}
		\]
		where $ I' $ and $ J' $ are Morita dg functors.
		
		By using the same proof as for Proposition~\ref{Prop: CY completion is Ginzburg}, $ \bm\Pi_{3}(\ca',\cb',\xi') $ is quasi-equivalent to $ \bm\Gamma'_{rel}=\bm\Gamma_{rel}(\tilde{\mu}_{v}(Q,F,W)) $. By the above construction, $ \bm{\Pi}_{2}(\cb) $ is Morita equivalent to $ \bm{\Pi}_{2}(F) $, $ \bm\Pi_{3}(\ca,\cb,\xi) $ is Morita equivalent to $ \bm\Gamma_{rel}=\bm\Gamma_{rel}(Q,F,W) $ and $ \bm{\Pi}_{2}(\cb')=\bm{\Pi}_{2}(F') $,  
		
		Thus, the following square of derived categories commutes up to isomorphism
		\[
		\begin{tikzcd}
			\cd(\bm{\Pi}_{2}(F'))\arrow[r,"\tilde{i'}^{*}"]\arrow[d,"I'^{*}"]&\cd(\bm\Gamma'_{rel})\arrow[d,"{J'^{*}\coloneqq\Psi_{+}}"]\\
			\cd(\bm{\Pi}_{2}(F))\arrow[r,"\tilde{i}^{*}"]&\cd(\bm\Gamma_{rel}),
		\end{tikzcd}
		\]
		where the triangle equivalence
		$$ \Psi_{+}:\cd(\bm\Gamma'_{rel})\ra\cd(\bm\Gamma_{rel}) ,$$
		sends the $ \bm\Gamma'_{i} $ to $ \mathbf\Gamma_{i} $ for $ i\neq v $ and $ \mathbf\Gamma_{v} $ to the cone
		$$ \cone(\bm\Gamma_{v}\ra\bigoplus_{\alpha}\bm\Gamma_{t(\alpha)})=\cone(\bm\Gamma_{v}\xrightarrow{(\alpha)}\displaystyle\bigoplus_{\alpha\in F_{1}:s(\alpha)=v}\bm\Gamma_{t(\alpha)}) $$ and the triangle equivalence
		$$ I'^{*}:\cd(\bm{\Pi}_{2}(F'))\ra\cd(\bm{\Pi}_{2}(F)) ,$$
		sends the $ e_{i}\bm{\Pi}_{2}(F') $ to $ e_{i}\bm{\Pi}_{2}(F) $ for $ i\neq v $ and $ e_{v}\bm{\Pi}_{2}(F') $ to the cone
		$$ \cone(e_{v}\bm{\Pi}_{2}(F)\xrightarrow{(\alpha)}\displaystyle\bigoplus_{\alpha\in F_{1}:s(\alpha)=v}e_{t(\alpha)}\bm{\Pi}_{2}(F)).$$

	In fact, the functor $ J'^{*} $ is the derived tensor product functor (see~\cite[Subsection 3.4]{kellerDerivedEquivalencesMutations2011})
	$$ ?\lten_{\bm\Gamma'_{rel}}U ,$$
	where the $ \bm\Gamma'_{rel}$-$\bm\Gamma_{rel} $-bimodule $ U $ is given by 
	$$ U=\displaystyle\bigoplus_{j\neq v}\bm\Gamma_{j}\oplus U_{v} $$ with $ U_{v}=\cone(\bm\Gamma_{v}\xrightarrow{(\alpha)}\displaystyle\bigoplus_{\alpha\in F_{1}:s(\alpha)=v}\bm\Gamma_{t(\alpha)}) $ and the functor $ I'^{*} $ is the derived tensor product functor
	$$ ?\lten_{\bm{\Pi}_{2}(F')}W ,$$
	where the $ \bm{\Pi}_{2}(F')$-$\bm{\Pi}_{2}(F) $-bimodule $ W $ is given by 
	$$ W=\displaystyle\bigoplus_{j\neq v}e_{j}\bm{\Pi}_{2}(F)\oplus W_{v} $$ with $ W_{v}=\cone(e_{v}\bm{\Pi}_{2}(F)\xrightarrow{(\alpha)}\displaystyle\bigoplus_{\alpha\in F_{1}:s(\alpha)=v}e_{t(\alpha)}\bm{\Pi}_{2}(F)) $.

	We next show that the functor $ I'^{*}\circ\mathrm{can}^{-1} $ is isomorphic to $ t_{S_{v}}^{-1} $. Let $ T $ be the shift cone
	$$ \Si^{-1}\cone(\bm{\Pi}_{2}(F)\xrightarrow{\mathrm{ev'}}\Hom_{k}(\RHom_{\bm{\Pi}_{2}(F)}(\bm{\Pi}_{2}(F),S_{v}),S_{v})), $$
	where $ \mathrm{ev'}(x)(f)=f(x) $ for any $ x\in\bm{\Pi}_{2}(F) $ and $ f\in\RHom_{\bm{\Pi}_{2}(F)}(\bm{\Pi}_{2}(F),S_{v}) $.

	By the 2-Calabi--Yau property, we have 
	\begin{equation*}
		\begin{split}
			\Hom_{k}(\RHom_{\bm{\Pi}_{2}(F)}(\bm{\Pi}_{2}(F),S_{v}),S_{v}))&\cong\Hom_{k}(D\RHom_{\bm{\Pi}_{2}(F)}(S_{v},\Si^{2}\bm{\Pi}_{2}(F),S_{v})\\
			&\cong S_{v}\ten_{k}\RHom_{\bm{\Pi}_{2}(F)}(S_{v},\Si^{2}\bm{\Pi}_{2}(F)).
		\end{split}
	\end{equation*}	
By \cite[Lemma 4.8]{GrantMarsh2017}, we have a natural isomorphism of functors
$$ \RHom_{\bm{\Pi}_{2}(F)}(S_{v},\Si^{2}\bm{\Pi}_{2}(F))\overset{\rm\textbf{L}}{\ten}_{\bm{\Pi_{2}}(F)}\,?\xrightarrow{}\RHom_{\bm{\Pi}_{2}(F)}(S_{v},\Si^{2}\,?) .$$
Therefore we get the following natural isomorphism of functors
$$ \Hom_{k}(\RHom_{\bm{\Pi}_{2}(F)}(\bm{\Pi}_{2}(F),S_{v}),S_{v})\overset{\rm\textbf{L}}{\ten}_{\bm{\Pi_{2}}(F)}\,?\ra\Hom_{k}(\RHom_{\bm{\Pi}_{2}(F)}(?,S_{v}),S_{v}) $$

Thus, for any object $ X $ in $ \cd(\bm{\Pi_{2}}(F)) $, $ T\overset{\rm\textbf{L}}{\ten}_{\bm{\Pi_{2}}(F)} X $ is determined by the following triangle
	$$ T\overset{\rm\textbf{L}}{\ten}_{\bm{\Pi_{2}}(F)} X\ra X\ra\Hom_{k}(\RHom_{\bm{\Pi}_{2}(F)}(\bm{\Pi}_{2}(F),S_{v}),S_{v})\overset{\rm\textbf{L}}{\ten}_{\bm{\Pi_{2}}(F)} X\xrightarrow{+1} ,$$
	which is equal to
	$$ T\overset{\rm\textbf{L}}{\ten}_{\bm{\Pi_{2}}(F)} X\ra X\ra\Hom_{k}(\RHom_{\bm\Pi_{2}(F)}(X,S_{v}),S_{v})\xrightarrow{+1} .$$
	This shows that we have a natural isomorphism of functors
	$$ T\overset{\rm\textbf{L}}{\ten}_{\bm{\Pi_{2}}(F)}(?)\iso t_{S_{v}}^{-1}(?)\colon\cd(\bm{\Pi_{2}}(F))\ra\cd(\bm{\Pi_{2}}(F)) .$$

	On the other hand, the vector space $ \RHom_{\bm{\Pi}_{2}(F)}(\bm{\Pi}_{2}(F),S_{v}) $ is isomorphic to $ k $. The map $ \mathrm{ev}' $ is just the following map
	\[
	\begin{tikzcd}
		\mathrm{ev'}:\bm{\Pi}_{2}(F)=\displaystyle\bigoplus_{i\neq v} e_{i}\bm{\Pi}_{2}(F)\oplus e_{v}\bm{\Pi}_{2}(F)\arrow[r,"{(0,\,\pi_{v})}"]&S_{v},
	\end{tikzcd}
	\]
	where $ \pi_{v} $ is the canonical projection morphism $ e_{v}\bm{\Pi}_{2}(F)\ra S_{v} $.

   By Subsection~\ref{subsection:Cofibrant resolutions of simples over a tensor algebra}, we have a short exact sequence in $ \cc(\bm\Pi_{2}(F)) $
		$$ 0\ra \ker(\pi_{v})\ra e_{v}\bm{\Pi}_{2}(F)\xrightarrow{\pi_{v}} S_{v}\ra0.$$
   Explicitly, we have
		$$ \ker(\pi_{1})=\bigoplus_{\rho\in Q_{1}:t(\rho)=v}\rho e_{s(\rho)}\bm{\Pi}_{2}(F)=r_{v}(e_{v}\bm{\Pi}_{2}(F))\oplus\bigoplus_{a\in F_{1}:s(a)=v}\tilde{a}(e_{t(a)}\bm{\Pi}_{2}(F)) $$ with the induced differential. Here $ r_{v} $ and $ \tilde{a} $ are the arrows in $ \bm\Pi_{2}(F) $ (see Definition~\ref{Def: derived preprojective algebra}). We use the assumption that $ v $ is a frozen source, so there are no arrows in $ F_{1} $ with target $ v $.
		
		Then it is easy to see that $ \ker(\pi_{v}) $ is isomorphic to the mapping cone of the morphism below
		$$ e_{v}\bm{\Pi}_{2}(F)\ra\bigoplus_{a\in F_{1}:s(a)=v}e_{t(a)}\bm{\Pi}_{2}(F) ,$$ where the corresponding component of the map is the left multiplication by $ a $. Thus, $ T $ is isomorphic to the dg bimodule $ \displaystyle\bigoplus_{i\neq v} e_{i}\bm{\Pi}_{2}(F)\oplus\cone (e_{v}\bm{\Pi}_{2}(F)\xrightarrow{(a)}\bigoplus_{a\in F_{1}:s(a)=v}e_{t(a)}\bm{\Pi}_{2}(F)) $. By the construction of $ I'^{*} $, the composition $ I'^{*}\circ\mathrm{can}^{-1} $ is exactly the tensor functor $ T\overset{\rm\textbf{L}}{\ten}_{\bm{\Pi_{2}}(F)}(?)\simeq t_{S_{v}}^{-1}(?) $.

	\end{proof}
	
	Similarly, if $ v $ is a frozen sink, we have the following dual of Theorem~\ref{Thm:categorification of ice mutaion1}.
	\begin{Thm}\label{Thm:categorification of ice mutaion2}
		Suppose that $ v $ is a frozen sink in $ Q $. Write $ (Q',F',W')=\tilde{\mu}_{v}(Q,F,W) $. Let $ \bm\Gamma_{rel}=\bm\Gamma_{rel}(Q,F,W) $ and $ \bm\Gamma'_{rel}=\bm\Gamma_{rel}(Q',F',W') $ be the complete relative Ginzburg dg algebras associated to $ (Q,F,W) $ and $ (Q',F',W') $ respectively. We have a triangle equivalence
		$$ \Psi_{-}:\cd(\bm\Gamma'_{rel})\ra\cd(\bm\Gamma_{rel}),$$
		which sends the $ \bm\Gamma'_{i} $ to $ \bm\Gamma_{i} $ for $ i\neq v $ and to the shifted cone
		$$ \Si^{-1}\cone(\bigoplus_{\alpha}\bm\Gamma_{s(\alpha)}\to\bm\Gamma_{v}),$$ where we have a summand $ \bm\Gamma_{s(\alpha)} $ for each arrow $ \alpha $ of $ F $ with target $ v $ and the corresponding component of the map is the left multiplication by $ \alpha $. The functor $ \Psi $ restricts to a triangle equivalence from $ \per(\bm\Gamma'_{rel}) $ to $ \per(\bm\Gamma_{rel}) $ and from $ \pvd(\bm\Gamma'_{rel}) $ to $ \pvd(\bm\Gamma_{rel}) $. Moreover, the following square commutes up to isomorphism
		\[
		\begin{tikzcd}
			\cd(\bm\Pi_{2}(F'))\arrow[r,"G'^{*}"]\arrow[d,swap,"\mathrm{can}"]&\cd(\bm\Gamma'_{rel})\arrow[d,"\Psi_{-}"]\\
			\cd(\bm\Pi_{2}(F))\arrow[d,swap,"t_{S_{v}}"]&\cd(\bm\Gamma_{rel})\arrow[d,equal]\\
			\cd(\bm\Pi_{2}(F))\arrow[r,swap,"G^{*}"]&\cd(\bm\Gamma_{rel}),
		\end{tikzcd}
		\]
		where $\mathrm{can}$ is the canonical functor which identifies $ \bm\Pi_{2}(F') $ with $ \bm\Pi_{2}(F) $ and $ t_{S_{k}} $ is the twist functor with respect to the $ 2 $-spherical object $ S_{v} $, which give rise to a triangle
		$$ \RHom_{\bm\Pi_{2}(F)}(S_{v},X)\ten_{k}S_{v}\ra X\ra t_{S_{v}}(X)\ra\Si \RHom_{\bm\Pi_{2}(F)}(S_{v},X)\ten_{k}S_{v} $$ for each object $ X $ of $ \cd(\bm\Pi_{2}(F)) $.	
	\end{Thm}


\begin{thebibliography}{PTW02}
	\bibitem{adachiDiscretenessSiltingObjects2019}
	Takahide Adachi, Yuya Mizuno, and Dong Yang, \emph{Discreteness of silting objects and $ t $-structures in triangulated categories}, Proc. Lond. Math. Soc. (3) 118 (2019), no. 1, 1–42.
		
	\bibitem{amiotUbiquityGeneralizedCluster2011a}
	Claire Amiot, Idun Reiten, and Gordana Todorov, \emph{The ubiquity of generalized cluster categories}, Adv. Math. 226 (2011), no. 4, 3813–3849.
	
	\bibitem{baurDimerModelsCluster2016}
	Karin Baur, Alastair D. King, and Robert J. Marsh, \emph{Dimer models and cluster categories of Grassmannians}, Proc. Lond. Math. Soc. (3) 113 (2016), no. 2, 213–260.
 
    \bibitem{vandenberghCalabiYauAlgebrasSuperpotentials2015}
    Michel Van den Bergh,
    \emph{Calabi--Yau algebras and superpotentials}, Selecta Math. (N.S.) 21 (2015), no. 2, 555–603.
    
    \bibitem{bernsteinCoxeterFunctorsGabriel1973}
    I. N. Bernstein, I. M. Gel’fand, and V. A. Ponomarev, \emph{Coxeter functors and Gabriel’s theorem}. Russ. Math. Surv., 28(2):17–32, 1973.
		
	\bibitem{bozecRelativeCriticalLoci2020}
	Tristan Bozec, Damien Calaque, and Sarah Scherotzke, \emph{Relative critical loci and quiver moduli}, to appear in Annales Scientifiques de l'ENS, arXiv:2006.01069.
	
	\bibitem{bravRelativeCalabiYau2019}
	Christopher Brav and Tobias Dyckerhoff, \emph{Relative Calabi–Yau structures}, Compos. Math. 155 (2019), no. 2, 372–412.
	
	\bibitem{christGinzburgAlgebrasTriangulated2021}
	Merlin Christ, \emph{Ginzburg algebras of triangulated surfaces and perverse schobers}, Forum Math. Sigma 10 (2022), Paper No. e8, 72 pp.
	
	\bibitem{derksenQuiversPotentialsTheir2008}
	Harm Derksen, Jerzy Weyman, and Andrei Zelevinsky, \emph{Quivers with potentials and their representations I: Mutations},
	Selecta Math. (N.S.) 14 (2008), no. 1, 59–119.
	
	
	\bibitem{fominClusterAlgebrasFoundations2002}
	Sergey Fomin and Andrei Zelevinsky, \emph{Cluster algebras I: Foundations}, J. Amer. Math. Soc., 15(2):497–529, 2002.
	
	
	\bibitem{fraserQuasihomomorphismsClusterAlgebras2016}
	Chris Fraser, \emph{Quasi-homomorphisms of cluster algebras}, Adv. in Appl. Math. 81 (2016), 40–77.
	
	\bibitem{fraserPositroidClusterStructures2020}
	Chris Fraser and Melissa Sherman-Bennett, \emph{Positroid cluster structures from relabeled plabic graphs}, Algebr. Comb. 5 (2022), no. 3, 469–513.
	
	\bibitem{gabrielCategoriesAbeliennes1962}
	P. Gabriel, \emph{Des cat\'egories ab\'eliennes}, Bull. Soc. Math. France 90 (1962), 323–448.
	
	\bibitem{ginzburgCalabiyauAlgebras2006a}
	Victor Ginzburg, \emph{Calabi--Yau algebras}, arxiv preprint math/0612139, 2006.
	
	
	\bibitem{GrantMarsh2017}
	Joseph Grant and Bethany Rose Marsh, \emph{
	Braid groups and quiver mutation},
	Pacific J. Math. 290 (2017), no. 1, 77–116.
	
	\bibitem{happelDerivedCategoryFinitedimensional1987}
	Dieter Happel, \emph{On the derived category of a finite-dimensional algebra}, Comment. Math. Helv. 62 (1987), no. 3, 339–389.
	
	\bibitem{kellerDerivingDGCategories1994a}
	Bernhard Keller, \emph{Deriving DG categories}, Ann. Sci. École Norm. Sup. (4) 27 (1994), no. 1, 63–102.
	
	\bibitem{kellerDifferentialGradedCategories2006}
	Bernhard Keller, \emph{On differential graded categories}, international Congress of Mathematicians. Vol. II, 151–190, Eur. Math. Soc., Zürich, 2006.
	
	\bibitem{kellerDeformedCalabiYau2011}
	Bernhard Keller, \emph{Deformed Calabi–Yau completions}, with an appendix by Michel Van den Bergh. J. Reine Angew. Math., 2011(654):125–180, 2011.
	
	\bibitem{kellerDerivedEquivalencesMutations2011}
	Bernhard Keller and Dong Yang, \emph{Derived equivalences from mutations of quivers with potential}, Adv. Math. 226 (2011), no. 3, 2118–2168.
	
	\bibitem{wuyilinMutationsFrozenVertices}
	Bernhard Keller and Yilin Wu, \emph{Relative cluster categories and Higgs categories with infinite-dimensional morphism spaces}, in preparation.
	
	\bibitem{labardini-fragosoQuiversPotentialsAssociated2009}
	Daniel Labardini-Fragoso, \emph{Quivers with potentials associated to triangulated surfaces}, Proc. Lond. Math. Soc. (3) 98 (2009), no. 3, 797–839.
	
	\bibitem{lodayCyclicHomology2013}
	Jean-Louis Loday, \emph{Cyclic Homology}, Grundlehren der Mathematischen Wissenschaften, vol.
	301, Springer, Berlin, 1992.
	
	\bibitem{positselskiTwoKindsDerived2011}
	Leonid Positselski, \emph{Two Kinds of Derived Categories, Koszul Duality, and Comodule-Contramodule Correspondence}, Mem. Amer. Math. Soc. 212 (2011), no. 996, vi+133 pp.
	
	\bibitem{presslandCalabiYauPropertiesPostnikov2019}
	Matthew Pressland, \emph{Calabi-Yau properties of Postnikov diagrams}, Forum Math. Sigma 10 (2022), Paper No. e56, 31 pp.
	
	\bibitem{presslandMutationFrozenJacobian2020}
	Matthew Pressland, \emph{Mutation of frozen Jacobian algebras}, J. Algebra 546 (2020), 236–273.
	
	\bibitem{seibergElectricmagneticDualitySupersymmetric1995}
	N. Seiberg, \emph{Electric-magnetic duality in supersymmetric non-Abelian gauge theories}, Nuclear Phys. B 435 (1995), no. 1-2, 129–146.
	
	\bibitem{toenDerivedAlgebraicGeometry2014a}
	Bertrand Toen, \emph{Derived algebraic geometry}, EMS Surv. Math. Sci. 1 (2014), no. 2, 153–240.
	
	\bibitem{wuyilinHiggs}
	Yilin Wu, \emph{Relative cluster categories and Higgs categories}, arXiv:2109.03707 [math.RT].
	

	
	\bibitem{yeungRelativeCalabiYauCompletions2016}
	Wai-kit Yeung, \emph{Relative Calabi–Yau completions}, arxiv:1612.06352 [math], 2016.
	
    \end{thebibliography}
\end{document}